\documentclass[11pt]{article}

\usepackage{comment,url,algorithm,algorithmic,graphicx,subcaption,relsize}
\usepackage{amssymb,amsfonts,amsmath,amsthm,amscd,dsfont,mathrsfs,mathtools,nicefrac}
\usepackage{float,psfrag,epsfig,color,xcolor,url,hyperref}
\usepackage{epstopdf,bbm,mathtools,enumitem}
\usepackage[toc,page]{appendix}
\usepackage[mathscr]{euscript}
\usepackage{xspace}

\usepackage[top=1in, bottom=1in, left=1in, right=1in]{geometry}

\def\balign#1\ealign{\begin{align}#1\end{align}}
\def\baligns#1\ealigns{\begin{align*}#1\end{align*}}
\def\balignat#1\ealign{\begin{alignat}#1\end{alignat}}
\def\balignats#1\ealigns{\begin{alignat*}#1\end{alignat*}}
\def\bitemize#1\eitemize{\begin{itemize}#1\end{itemize}}
\def\benumerate#1\eenumerate{\begin{enumerate}#1\end{enumerate}}

\newenvironment{talign*}
 {\csname align*\endcsname}
 {\endalign}
\newenvironment{talign}
 {\csname align\endcsname}
 {\endalign}

\def\balignst#1\ealignst{\begin{talign*}#1\end{talign*}}
\def\balignt#1\ealignt{\begin{talign}#1\end{talign}}

\let\originalleft\left
\let\originalright\right
\renewcommand{\left}{\mathopen{}\mathclose\bgroup\originalleft}
\renewcommand{\right}{\aftergroup\egroup\originalright}

\def\Holder{H\"older\xspace}

\def\tinycitep*#1{{\tiny\citep*{#1}}}
\def\tinycitealt*#1{{\tiny\citealt*{#1}}}
\def\tinycite*#1{{\tiny\cite*{#1}}}
\def\smallcitep*#1{{\scriptsize\citep*{#1}}}
\def\smallcitealt*#1{{\scriptsize\citealt*{#1}}}
\def\smallcite*#1{{\scriptsize\cite*{#1}}}

\def\mbb#1{\mathbb{#1}}

\def\R{\mathbb{R}}

\def\<{\left\langle} %
\def\>{\right\rangle}

\def\E{\mbb{E}} %

\def\P{\mbb{P}} %

\DeclareSymbolFont{rsfs}{U}{rsfs}{m}{n}
\DeclareSymbolFontAlphabet{\mathscrsfs}{rsfs}

\ifdefined\nonewproofenvironments\else
\ifdefined\ispres\else
\newtheorem{theorem}{Theorem}
\newtheorem{lemma}[theorem]{Lemma}

\renewenvironment{proof}{\noindent\textbf{Proof.}\hspace*{.3em}}{\qed\\}
\newenvironment{proof-sketch}{\noindent\textbf{Proof Sketch}
  \hspace*{1em}}{\qed\bigskip\\}
\newenvironment{proof-idea}{\noindent\textbf{Proof Idea}
  \hspace*{1em}}{\qed\bigskip\\}
\newenvironment{proof-of-lemma}[1][{}]{\noindent\textbf{Proof of Lemma {#1}}
  \hspace*{1em}}{\qed\\}
\newenvironment{proof-of-theorem}[1][{}]{\noindent\textbf{Proof of Theorem {#1}}
  \hspace*{1em}}{\qed\\}
\newenvironment{proof-attempt}{\noindent\textbf{Proof Attempt}
  \hspace*{1em}}{\qed\bigskip\\}

\fi

\newtheorem{proposition}[theorem]{Proposition}

\newtheorem{assumption}{Assumption}
\fi
\makeatletter
\@addtoreset{equation}{section}
\makeatother

\hypersetup{
  colorlinks,
  linkcolor={red!50!black},
  citecolor={blue!50!black},
  urlcolor={blue!80!black}
}

\mathtoolsset{showonlyrefs}

\allowdisplaybreaks[4]

\newcommand{\rmd}{\mathrm d}
\newcommand{\Ltwo}[1]{\left\|#1\right\|_{L^2}}

\newcommand{\bigo}{\mathcal O}

\newcommand{\G}{\mathcal G}

\def\Ltwo{\mathbb L_2}
\newcommand{\law}{\mathcal L}

\definecolor{darkmidnightblue}{rgb}{0.0, 0.2, 0.4}
\definecolor{darkpowderblue}{rgb}{0.0, 0.2, 0.6}
\definecolor{dukeblue}{rgb}{0.0, 0.0, 0.61}

\hypersetup{
    colorlinks = true,
    citecolor= midnightblue,
    urlcolor= black,
    breaklinks=true,
    linkcolor = midnightblue,
    linkbordercolor = {white},
}

\definecolor{darkmidnightblue}{HTML}{003366}    
\definecolor{midnightblue}{HTML}{0059b3}
\definecolor{chromered}{HTML}{f14233}

\begin{document}

\title{
Diffusion Models with Heavy-Tailed Targets: Score Estimation and Sampling Guarantees}

 \author{
 Yifeng Yu\thanks{Department of Mathematical Sciences, Tsinghua University \texttt{yyf22@mails.tsinghua.edu.cn}}
 \and
 Lu Yu\thanks{
  Department of Data Science,
  City University of Hong Kong \texttt{lu.yu@cityu.edu.hk}
 }
}

\maketitle

\begin{abstract}
Score-based diffusion models have become a powerful framework for generative modeling, with score estimation as a central statistical bottleneck. 
Existing guarantees for score estimation largely focus on light-tailed targets or rely on restrictive assumptions such as compact support, which are often violated by heavy-tailed data in practice.
In this work, we study conventional (Gaussian) score-based diffusion models when the target distribution is heavy-tailed and belongs to a Sobolev class with smoothness parameter $\beta>0$. 
We consider both exponential and polynomial tail decay, indexed by a tail parameter $\gamma$.
Using kernel density estimation, we derive sharp minimax rates for score estimation, revealing a qualitative dichotomy: under exponential tails, the rate matches the light-tailed case up to polylogarithmic factors, whereas under polynomial tails the rate depends explicitly on $\gamma$.
We further provide sampling guarantees for the associated continuous reverse dynamics. 
In total variation, the generated distribution converges at the minimax optimal rate $n^{-\beta/(2\beta+d)}$ under exponential tails (up to logarithmic factors), and at a $\gamma$-dependent rate under polynomial tails. Whether the latter sampling rate is minimax optimal remains an open question.
These results characterize the statistical limits of score estimation and the resulting sampling accuracy for heavy-tailed targets, extending diffusion theory beyond the light-tailed setting.

\end{abstract}

 \tableofcontents

\section{Introduction}

Diffusion models have become a central tool in modern generative modeling, with impressive performance in image generation~\cite{dhariwal2021diffusion,ho2020denoising,ramesh2022hierarchical,song2020denoising}, natural language processing~\cite{popov2021grad}, and computational biology~\cite{anand2022protein,xu2022geodiff}. A commonly used formulation is the score-based generative model (SGM), which is implemented through stochastic differential equations~\cite{song2020score}. 
At a high level, SGMs consist of a forward process that gradually perturbs data drawn from the target distribution into noise, and a reverse process that aims to reconstruct the target distribution from noise. 
The reverse dynamics depend on the unknown score function, which specifies the local direction of denoising at each noise level. 
Consequently, score estimation controls the accuracy of the reverse diffusion process and is widely regarded as the core statistical bottleneck in SGMs.

Recently, a substantial line of research has investigated how accurately the score function can be learned within the diffusion model framework. 
Early work~\cite{block2020generative} establishes score estimation guarantees under the $\ell_2$-norm, though the corresponding bounds depend on Rademacher complexities of some unknown hypothesis classes. Subsequent progress in~\cite{gupta2024improved} demonstrates exponential improvements in sample complexity by employing sufficiently expressive neural network architectures. 
More recent studies further advance score estimation using neural networks. 
In particular, \cite{cole2024score} shows that for sub-Gaussian target distributions, the score function can be locally approximated by a neural network whose parameters admit uniform bounds, thereby overcoming certain high-dimensional challenges. \cite{han2024neural} derives the explicit sample complexity guarantees for score estimation with neural networks, leveraging tools from the neural tangent kernel method. 
Complementary to these learnability results, \cite{chewi2025ddpm} reduces DDPM~\cite{ho2020denoising} score estimation to classical distribution learning tasks such as PAC density estimation and uses this reduction to prove computational lower bounds including cryptographic hardness for score estimation of general Gaussian mixture models.

Other works have explored score estimation under additional structural assumptions on the target distribution.
For instance,~\cite{gatmiry2024learning} and~\cite{chen2024learning} show that when the data follow a Gaussian mixture distribution, the score function can be learned efficiently using piecewise low-degree polynomial approximations.
Another line of research considers data supported on low-dimensional structures. 
In this vein,~\cite{chen2023score} shows that neural networks can accurately approximate the score when the data lie in an unknown linear subspace, while~\cite{azangulov2024convergence} establishes accurate score approximation for data supported on a bounded low-dimensional manifold via suitably designed neural networks. 
More recently,~\cite{yakovlev2025generalization} and~\cite{yakovlev2025implicit} study denoising and implicit score matching under relaxed manifold assumptions and derive nonasymptotic approximation and generalization bounds with rates governed by the intrinsic dimension.

A separate line of work develops minimax convergence rates for SGMs, characterizing the best achievable sampling error as a function of the sample size and ambient dimension.
However, existing analyses rely on strong regularity assumptions on the target distribution, such as sub-Gaussian tails~\cite{stephanovitch2025generalization,zhang2024minimax,cai2025minimax,dou2024optimal}, explicit lower bounds on the density~\cite{oko2023diffusion,fan2025optimal}, or compact support~\cite{stephanovitch2025generalization,dou2024optimal,tang2024adaptivity}. 
Such conditions are restrictive and fail to hold for many heavy-tailed distributions encountered in practice, for example, in finance~\cite{bradley2003financial,harvey2013dynamic}, signal and image processing~\cite{achim2003sar,briassouli2005hidden}, and various scientific domains involving measurement errors with non-negligible tail mass~\cite{adler1998practical}.
As a result, current minimax theories do not apply to heavy-tailed targets and leave open fundamental questions about the behavior and limits of SGMs in this setting.

In this work, we address this gap by focusing on score estimation in conventional score-based diffusion models when the target distribution is heavy-tailed. Since the score function governs the reverse diffusion dynamics, understanding its statistical complexity is essential for the convergence analysis of diffusion models.
Motivated by this, we study the following fundamental questions:
\begin{center}
\textit{
1) Under the standard score-based diffusion pipeline with Gaussian noising, what are the minimax optimal rates for estimating the score when the underlying data distribution is heavy-tailed?
}
\end{center}
\begin{center}
\textit{
2) How do these score estimation rates translate into bounds on the sampling error between the generated and target distributions?
}
\end{center}
We provide explicit answers to both questions.
We establish the first sharp minimax guarantees for score estimation under heavy-tailed target distributions. 
Specifically, we consider densities in a Sobolev class with smoothness parameter $\beta>0$ and analyze two representative tail regimes: (i) exponential tail decay, which includes the sub-Gaussian case, and (ii) polynomial tail decay, corresponding to heavier-tailed distributions. Both regimes are indexed by a tail parameter $\gamma$. 
For each regime, we construct kernel-based score estimators and show that they attain the corresponding minimax optimal rates. 
We further analyze the associated continuous reverse diffusion process and derive convergence guarantees for the resulting sampling distribution. 
In particular, in the exponential decay regime, we show that the generated distribution attains the minimax optimal sampling rate in total variation. Our main contributions can be summarized as follows.
\begin{itemize}
    \item 
    We analyze the conventional Gaussian forward diffusion and establish a nearly minimax optimal rate for score estimation in the polynomial tail regime.
    In particular, when the target density has polynomial decay with tail index $\gamma$, we show that a kernel-based estimator achieves the rate {$n^{-\frac{\gamma+1}{d+\gamma+1}}\,t^{-1-\frac{d}{2}\frac{\gamma+1}{d+\gamma+1}}$} (Theorem~\ref{thm:MSEst}),
thereby extending minimax score estimation theory for SGMs beyond light-tailed settings.
    \item 
    We further study the associated continuous reverse diffusion  driven by the learned score and derive, to our knowledge, the first total variation sampling error bound {$\bigo \Big(n^{-\frac{2\beta(\gamma+1)}{4\beta(d+\gamma+1)+d\left(d+2(\gamma+1)\right)}}\Big)$} in the polynomial tail regime (Theorem~\ref{thm:poly-sampling}).

    \item For exponentially decaying targets, we obtain minimax optimal guarantees at both stages of the pipeline: score estimation under the Gaussian forward process achieves the optimal rate {$n^{-1}t^{-1-d/2}$} (Theorem~\ref{thm:mse_exp}), and the continuous-time reverse diffusion driven by the learned score attains the optimal rate {$n^{-\beta/(2\beta+d)}$} in total variation (Theorem~\ref{thm:exp_sampling}), matching known results in the sub-Gaussian targets.
\end{itemize}
Taken together, these results provide a unified statistical perspective on score estimation and the resulting sampling error for standard score-based diffusion models under a broad class of heavy-tailed targets. They also reveal a sharp dichotomy between exponential and polynomial tail decay. For exponentially decaying targets, the conventional Gaussian diffusion framework is statistically robust, leading to rates consistent with those in the sub-Gaussian regime. In contrast, for polynomially decaying targets, the minimax score estimation rate under the same Gaussian forward process depends explicitly on the tail index, and this tail-dependent penalty propagates to the sampling error bound, showing that tail heaviness fundamentally increases the statistical difficulty.

\subsection{More Related Work}

\textbf{Convergence Analysis for Diffusion Models.} 
Motivated by the empirical success of score-based diffusion models, a growing literature has examined their convergence properties. Most existing analyses assume access to accurate score estimates and focus on the convergence behavior of the resulting sampling dynamics. These works typically rely on bounded second-moment assumptions on the target distribution (e.g., \cite{cordero2025non,chen2022sampling,yu2025advancing,gao2025wasserstein,conforti2025kl,chen2023improved}) or impose compact support conditions (e.g., \cite{de2022convergence,lee2023convergence,azangulov2024convergence}). While these studies offer valuable insights into the stability, discretization error, and convergence behavior of diffusion-based samplers, they treat the score function as known and therefore bypass the fundamental statistical challenge of learning the score from finite data.

\vspace{0.1cm}

\vspace{0.2cm}
\noindent \textbf{Diffusion Models Beyond Gaussian Noise.}
A separate thread of work considers SGMs that replace Gaussian perturbations in the forward process with non-Gaussian noise, such as generalized Gaussian distribution~\cite{deasy2021heavy}, gamma distribution~\cite{nachmani2021denoising}, $t$-distributions~\cite{pandey2024heavy} or $\alpha$-stable Lévy processes~\cite{yoon2023score,shariatianheavy}. 
These approaches share a common motivation: to enhance robustness in settings with imbalanced or heavy-tailed data and to broaden the applicability of diffusion models beyond the Gaussian framework.
However, theoretical development in this direction has been limited. Unlike the Gaussian case, non-Gaussian perturbations typically do not yield closed-form transition densities or tractable reverse SDEs, making it difficult to analyze the resulting dynamics or characterize sampling error. As a consequence, most of these works provide no explicit convergence guarantees, nor do they study the statistical difficulty of estimating the corresponding score functions.

\vspace{0.2cm}
\noindent\textbf{Notation.}
Let $\mathbb{R}^d$ be the $d$-dimensional Euclidean space.
$L^1(\R^d)$ denotes the space of integrable functions on $\R^d$. 
For $p \geqslant 1$, define the $L^p$-norm as $\|f\|_{L^p}:=\left(\int_{\R^d}|f(x)|^p\,\rmd x\right)^{1/p}$.
Given a measurable weight $q:\mathbb{R}^d\to[0,\infty)$, define the weighted norm $\|f\|_{L^p(q)}:=\left(\int_{\R^d}|f(x)|^pq(x)\,\rmd x\right)^{1/p}$. 
For any random vector $X$, let $\mathcal{L}(X)$ denote its law, and define $\|X\|_{\Ltwo}:=\sqrt{\E[\|X\|^2]},$ where $\|\cdot\|$ denotes the Euclidean norm.
We use $\text{polylog}(n)$ as shorthand for $(\log n)^{C}$ with some constant $C>0$.
For $f,g:\R^d\to \R$, define the convolution $(f*g)(x)=\int_{\mathbb R^d} f(x-y)g(y)\rmd y$.
Define the Fourier transform for the function $f:\R^d\to\R$ as
$
\mathcal F[f](\omega)=\int_{\R^{d}} f(x)\,e^{-i\langle \omega,x\rangle}\,\rmd x,
$
where $\langle \omega,x\rangle$ is the standard inner product in $\R^{d}$.
For a multi-index $\alpha=(\alpha_1,\ldots,\alpha_d)\in\mathbb N_0^{d}$ and $\omega=(\omega_1,\ldots,\omega_d)\in\R^{d}$, we define
$
\omega^\alpha:=\prod_{i=1}^{d}\omega_i^{\alpha_i}.
$
$\operatorname{TV}(p,q)$ and $\mathrm{KL}(p\|q)$ denote the total variation distance and the Kullback--Leibler divergence between probability measures \(p\) and \(q\), respectively.

\section{Preliminaries}
In this section, we present the SGMs studied in this work and state the main assumptions.
\subsection{Score-Based Diffusion Models}
\textbf{Framework.} \quad We consider the forward process
\begin{equation}
\label{eq:forward0}
\rmd X_t= f(X_t,t)\rmd t+g(X_t,t)\rmd B_t\,,
\end{equation}
where the initial point $X_0\sim  p_0$ follows the target distribution, and $B_t$ denotes the standard $d-$dimensional Brownian motion.
Here, the drift $f: \mathbb{R}^d\times \mathbb{R}_+\to \mathbb{R}^d$ and the function $g:\mathbb{R}^d\times \mathbb{R}_+\to \mathbb{R}^{d\times d}$ are diffusion parameters.
To ensure the SDE~\eqref{eq:forward} is well defined, certain conditions on the pair $(f,g)$ are needed. 
Common choices of this pair depend on modeling goals, and we refer~\cite{tang2024score} for a detailed overview.
For clarity, we adopt the simplest possible choice in this work by setting $f(X_t,t)=0$ and $g(X_t,t)=I$.
Under this choice, the forward process reduces to a standard Brownian motion
\begin{equation}
    \rmd X_t=\rmd B_t\,.
    \label{eq:forward}
\end{equation}
The solution $(X_t)_{t\in [0,T]}$ admits a closed-form conditional distribution given $X_0$, namely
\begin{align*}
    X_t=X_0+\sqrt{t}Z, \quad Z\sim\mathcal N(0,I_d)\,.
\end{align*}
{Consequently, by perturbing the original data $X_0\sim p_0$ with Gaussian noise $B_t$ over a sufficiently large time horizon $T$, the marginal distribution $p_T$ of $X_T$ becomes only weakly dependent on $X_0$ and is approximately Gaussian when $p_0$ is light-tailed (e.g. sub-Gaussian). In contrast, when $p_0$ is heavy-tailed, Gaussian convolution smooths the distribution but does not eliminate heavy-tail behavior. 
In this regime, $X_T$ is close to Gaussian only in the high-probability region, while deviations persist in the tails.

Then, diffusion models generate new data by reversing the SDE~\eqref{eq:forward}, which leads to the following backward SDE
\begin{equation}
\rmd X_t^{\leftarrow} =\nabla \log p_{T-t}(X_t^{\leftarrow})\rmd t + \rmd W_t\,,
\label{eq:backward}
\end{equation}
where $X^\leftarrow_0\sim p_T$, and the term $\nabla \log p_t$, referred to as the \textit{score function} for $p_t$.
Additionally, $W_t$ denotes another standard Brownian motion independent of $B_t$.
Under mild conditions, when initialized at
$X^\leftarrow_0\sim p_T$, the backward process $\{X^\leftarrow_t\}_{0\leqslant t\leqslant T}$ has the same distribution as the forward process $\{X_{T-t}\}_{0\leqslant t\leqslant T}$~\cite{anderson1982reverse,cattiaux2023time}. 
As a result, running the reverse diffusion $X^{\leftarrow}_t$ from $t=0$ to $T$ will generate a sample from the target data distribution $p_0$.
{Note that the density $p_T$ is unknown, one can approximate it using the Gaussian distribution $\hat p_T= \mathcal{N}({0},TI_d)$\footnote{We note that, even when $p_0$ is heavy-tailed, initializing with $\hat p_T=\mathcal{N}(0,TI_d)$ is accurate in the bulk for large $T$. Any remaining discrepancy in the far tails can be viewed as a terminal initialization error.}.}Therefore, we derive a reverse diffusion process defined by
\begin{align}
    \label{eq:Yt}
    \rmd Y_t=\nabla\log p_{T-t}(Y_t)\rmd t+\rmd W_t,\quad Y_0\sim \hat{p}_T\,.
\end{align}
\noindent \textbf{Score Matching.} \quad Another  challenge in working with~\eqref{eq:backward} is that the score function $\nabla \log p_t$ is unknown, as the distribution $p_t$ is not explicitly available.
In practice, the score is approximated from data by learning a surrogate $s_\theta$ through the score-matching objective
\begin{align*}
\underset{\theta\in\Theta}{\text{minimize}}~~~ \E[\|s_\theta(t,X_t)-\nabla \log p_t(X_t)\|^2]\,,
\end{align*}
where $\{s_\theta:\theta\in\Theta\}$ is a sufficiently expressive function class.
This optimization can be approached using either parametric estimators based on neural networks or nonparametric methods such as kernel density score estimators. 
While neural networks dominate practical implementations, kernel methods are often adopted in theoretical analyses due to their tractable statistical properties. 
In this work, we adopt the latter perspective ad propose a kernel-based estimator of the score function.

Substituting the learned score estimate $\hat s$ into the backward process~\eqref{eq:backward}, we obtain the following practical continuous-time backward SDE,
\begin{equation}
\rmd\hat Y_t= \hat s_{T-t}(\hat Y_t)\rmd t + \rmd W_t\,,\quad \hat Y_0\sim \hat p_T
\label{eq:backward1}
\end{equation}
Our analysis focuses on score estimation and on bounding the total variation distance between the distribution of the generated sample $\hat Y_t$ and the target $p_0.$
The discretization of this continuous reverse SDE is beyond the scope of this work. In practice, this SDE is implemented using numerical schemes, and the associated discretization bias has been studied for Euler–Maruyama (e.g.~\cite{chen2023improved,gao2025wasserstein}), exponential integrators (e.g.~\cite{hochbruck2010exponential,zhang2022fast}), randomized midpoint schemes~\cite{gupta2024faster,li2024improved,yu2025advancing}, and higher order Runge-Kutta methods~\cite{huang2025fast,huang2025convergence,wu2024stochastic}.

\subsection{Modeling Setting and Regularity Conditions}
To derive rigorous bounds for both score estimation and sampling error in the heavy-tailed regime, we impose structural assumptions on both the tail decay and the regularity of the target distribution $p_0$. 
In particular, we consider two types of tail behavior, namely exponential decay and polynomial decay, and we focus on target densities belonging to a Sobolev smoothness class, as detailed below.

\begin{assumption}[Exponential Tail Decay] \label{asm:heavytail}
    The density $p_0$ satisfies the following decay condition for constants $c_1,C_1 > 0$:
    \begin{align*}
         p_0(x)\leqslant C_1\exp(-c_1\|x\|^\gamma),\quad \gamma\in(0,\infty).
    \end{align*}
\end{assumption}

The parameter $\gamma$ governs the tail behavior of the target distribution. When $\gamma=2$, $p_0$ is Gaussian and exhibits light tails with finite moments of all orders. The case $\gamma=1$ corresponds to Laplace-type distributions, which retain exponential decay but have heavier tails than the Gaussian. For $\gamma\in(0,1)$, the decay is slower, leading to heavier-tailed distributions for which higher-order moments may no longer exist.

\begin{assumption}[Polynomial Tail Decay]\label{asm:p0tail}
    The target distribution $p_0$ satisfies the following exponential decay condition:  there exists a constant $C_2 > 0$ such that
    \begin{align*}
        p_0(x)\leqslant \dfrac{C_2}{(1+\|x\|^2)^{(1+\gamma+d)/2}}, \quad \gamma\in(0,\infty).
    \end{align*}
\end{assumption}
The polynomial tail decay condition encompasses a wide range of heavy-tailed distributions encountered in practice. Notable examples include multivariate Student's $t$ distributions, the Pareto family, and finite mixtures with heavy–tailed components. 
Such distributions may exhibit infinite variance and may fail to possess finite moments of higher order, making this regime substantially broader and more challenging than the exponential tail setting.

Beyond tail behavior, we require mild smoothness conditions on $p_0$ to control the approximation error of the kernel density estimator.

\begin{assumption}[Sobolev Regularity]\label{asm:p0smooth}
    The true distribution $p_0$ belongs to the Sobolev class $\mathcal P_{\mathcal S}(\beta,L)$ with smoothness parameter $\beta>0$~\cite{tsybakov2008nonparametric}. Specifically, for $\beta,L\in\R_+$,
\begin{align*}
\mathcal P_{\mathcal S}(\beta,L)
=\Bigl\{p\in L^1(\R^d)\ \Big|\ &p\geqslant 0,\ \int_{\R^d} p(x)\,\rmd x=1,\\
&\forall \alpha\in\mathbb N_0^d\ \text{with}\ \sum_{i=1}^d \alpha_i=\beta,\ 
\int_{\R^d}|\omega^\alpha|^2\,|\mathcal F[p](\omega)|^2\,\rmd \omega\leqslant (2\pi)^dL^2\Bigr\}.
\end{align*}
\end{assumption}
This assumption is also adopted in prior minimax theory of score estimation for diffusion models~\cite{zhang2024minimax}.

\subsection{Kernel-Based Score Estimation}
\label{sec:kde}
To obtain a statistically tractable estimator of the score
$s_t(x)=\nabla \log p_t(x)$, we use a nonparametric construction based on
kernel density estimation (KDE). Given $n$ i.i.d.\ samples $\{X^{(i)}\}_{i=1}^n$
from the data distribution $p_0$, we first estimate the marginal density $p_t=\varphi_t * p_0$ using the kernel $\varphi_t$, and then form a
plug-in estimator of $\nabla \log p_t$. To avoid numerical instability in low-density regions, we additionally apply a threshold to the density estimate.

\paragraph{Gaussian KDE and score estimator.}
We define the Gaussian kernel density estimator of \(p_t\) and its gradient by
\begin{align*}
    \hat p_t(x)= \frac{1}{n}\sum_{i=1}^n \varphi_t\!\left(x-X^{(i)}\right), \qquad 
    \nabla \hat p_t(x)
    =  \frac{1}{nt}\sum_{i=1}^n \left(X^{(i)}-x\right)\varphi_t\!\left(x-X^{(i)}\right),
\end{align*}
where $\varphi_t(x)=\frac{1}{(2\pi t)^{d/2}}\exp\left(-\|x\|^2/2t\right)$ is the Gaussian transition density.
In contrast to standard KDE, the bandwidth is not a tuning parameter here; instead, the diffusion time 
$t$ naturally serves as the bandwidth.

We then define the thresholded score estimator
\begin{align}\label{eq:hatst_rho}
    \hat s_t(x):=\dfrac{\nabla \hat p_t(x)}{\hat p_t(x)}\mathbbm 1_{\{\hat p_t(x)\geqslant \rho_n\}}(x),
    \qquad
    \rho_n := \frac{\log n}{n(2\pi t)^{d/2}}\,,
\end{align}
so that the denominator is bounded away from zero on the region where the ratio
is evaluated.

\paragraph{Unbiasedness.}
Conditioned on the dataset, $\hat p_t$ and $\nabla \hat p_t$ are unbiased
estimators of $p_t$ and $\nabla p_t$, respectively. In particular,
\begin{align*}
    \E[\hat p_t(x)]&=\int_{\R^d}\dfrac{1}{(2\pi t)^{d/2}}\exp\left(-\dfrac{\|x-y\|^2}{2t}\right)p_0(y)\,\rmd y
    =(\varphi_t*p_0)(x)
    =p_t(x)
\end{align*}
and
\begin{align*}
\mathbb{E}\,\!\left[\nabla \hat p_t(x)\right]
&= \int_{\mathbb{R}^d} \nabla_x \varphi_t(x-y)\,p_0(y)\,\mathrm{d}y = \nabla_x \int_{\mathbb{R}^d} \varphi_t(x-y)\,p_0(y)\,\mathrm{d}y
 = \nabla p_t(x).
\end{align*}
This unbiasedness is central to our error analysis and will be used repeatedly throughout the heavy-tailed setting we consider.

\section{Main Results: Convergence Analysis}

In this section, we establish theoretical guarantees for the proposed score estimators and the resulting sampling procedure under the standard diffusion framework described in the previous section, assuming the target density $p_0$ is heavy-tailed.
We first treat the case of polynomially decaying tails, and then turn to polynomially exponentially decaying tails.
We include a proof sketch in Section~\ref{sec:sketch} that outlines the overall structure of the main proofs, highlights the key technical obstacles, and clarifies how our setting and arguments differ from prior work. The complete proofs are deferred to the appendix.

\subsection{Polynomial decay}
\label{sec:poly}

In this part, we first analyze the score estimation error, and then derive a sampling error bound for the distribution generated by the reverse 
SDE~\eqref{eq:backward1} driven by the proposed score estimator, assuming the target density $p_0$ is heavy-tailed with polynomial decay.

We now state our main statistical guarantee for score estimation at diffusion time $t\in(0,T]$. The following theorem provides a mean squared error bound for the proposed estimator $\hat s_t$.
\begin{theorem}
\label{thm:MSEst}
Suppose that $p_0$ satisfies Assumption~\ref{asm:p0tail} and~\ref{asm:p0smooth}. For any $t\in(0,T]$, the proposed estimator $\hat s_t$ satisfies
\begin{align*}
    \E\left[\int_{\R^d}\|\hat s_t(x)-s_t(x)\|^2p_t(x)\,\rmd x\right]\lesssim {\rm{polylog}}(n)n^{-\frac{\gamma+1}{d+\gamma+1}} t^{-1-\frac{d}{2}\frac{\gamma+1}{d+\gamma+1}}\,.
\end{align*}
\end{theorem}
The MSE bound in Theorem~\ref{thm:MSEst} worsens as $t$ decreases, reflecting the fact that score estimation becomes statistically ill-conditioned at very small noise levels. For downstream sampling guarantees, however, the relevant quantity is not the pointwise error at a fixed $t$, but its accumulation over time along the diffusion trajectory. We therefore introduce a lower cutoff $t_0\in(0,T)$ and bound the integrated error over $[t_0,T]$ in the following theorem.

\begin{theorem}\label{thm:t0scoreerror}
    Suppose that $p_0$ satisfies Assumption~\ref{asm:p0tail} and~\ref{asm:p0smooth}.
    Then, for any $0<t_0<T$,
    \begin{align*}
        \int_{t_0}^T\E\left[\int_{\R^d}\|\hat s_t(x)-s_t(x)\|^2p_t(x)\,\rmd x\right]\,\rmd t\lesssim \operatorname{polylog}(n)n^{-\frac{\gamma+1}{d+\gamma+1}}t_0^{-\frac{d(\gamma+1)}{2(d+\gamma+1)}}\,.
    \end{align*}
\end{theorem}
As a point of comparison, existing analyses of score estimation in SGMs often focus on the light-tailed regime, such as sub-Gaussian or compact support assumptions.
Our results interpolate between heavy and light tails through the tail index $\gamma$. 
As $\gamma$ increases, the polynomial tail condition becomes less restrictive, and the tail dependent penalty in our bounds vanishes. 
In particular,
\[
\lim_{\gamma\to\infty}\frac{\gamma+1}{d+\gamma+1}=1.
\]
Accordingly, the sample size dependence in Theorem~\ref{thm:t0scoreerror} approaches \(n^{-1}\) (up to logarithmic factors), which is consistent with known rates for score estimation under sub-Gaussian assumptions of $p_0$~\cite{zhang2024minimax}.

In fact, Theorem~\ref{thm:t0scoreerror} controls the score estimation error integrated over diffusion times $t\in[t_0,T]$. 
The dependence on $t_0$ is unavoidable: as $t\to 0$, the marginal density $p_t$ becomes increasingly irregular in the tails, and the plug-in estimator $\hat s_t$ suffers from a variance inflation due to the scarcity of samples in low-density regions. 
This suggests treating $t_0>0$ as a regularization parameter and avoiding the small noise regime $t<t_0$ in downstream sampling guarantees. 
Doing so introduces an additional bias because the reverse time dynamics are no longer driven all the way back to \(p_0\), but rather to the smoothed distribution $p_{t_0}=p_0*\varphi_{t_0}$. 
The next theorem quantifies this early stopping bias in total variation distance.

\begin{theorem}\label{thm:early_stopping}
Under Assumptions~\ref{asm:p0tail} and~\ref{asm:p0smooth} with Sobolev smoothness \(\beta\in(0,2]\), define \(p_{t_0}=p_0*\varphi_{t_0}\).
Then, the early stopping bias is bounded in total variation as
\begin{align*}
\operatorname{TV}(p_0, p_{t_0}) \lesssim t_0^{\frac{\beta(\gamma+1)}{d + 2(\gamma+1) + 2\beta}}\,.
\end{align*}
\end{theorem}
Theorem~\ref{thm:early_stopping} shows that early stopping at time $t_0$ incurs a controllable approximation error $\operatorname{TV}(p_0,p_{t_0})$, which vanishes as $t_0\to 0$ at a rate governed by the Sobolev smoothness $\beta$ and the tail index $\gamma$. 
Combined with Theorem~\ref{thm:t0scoreerror}, this yields a standard bias--variance tradeoff: choosing a smaller \(t_0\) reduces the early stopping bias but increases the score estimation error through the factor \(t_0^{-\frac{d(\gamma+1)}{2(d+\gamma+1)}}\). 
This tradeoff motivates selecting \(t_0\) to balance the two terms in the final sampling bound as follows.
\begin{theorem}
\label{thm:poly-sampling}
Assume $p_0$ satisfies Assumptions~\ref{asm:p0tail} and~\ref{asm:p0smooth}, with Sobolev smoothness parameter $\beta\in(0,2]$.
Set
$$
t_0 \;=\; n^{-\frac{2\left(d+2(\gamma+1)\right)}{4\beta(d+\gamma+1)+d\left(d+2(\gamma+1)\right)}}
\quad \text{ and }\quad
T=n^{\frac{4\beta(\gamma+1)}{4\beta(d+\gamma+1)+d(d+2(\gamma+1))}}.
$$
Then, there exists a score estimator $(\hat s_t)_{t\in[0,T]}$ such that the sample
$\hat Y_{T-t_0}$ from the corresponding reverse time diffusion process~\eqref{eq:backward1} satisfies
\begin{align*}
\E\,\left[\operatorname{TV} \left(p_0,\law(\hat Y_{T-t_0})\right)\right]
\;\lesssim\;
\operatorname{polylog}(n)\,
n^{-\frac{2\beta(\gamma+1)}{4\beta(d+\gamma+1)+d\left(d+2(\gamma+1)\right)}} \,.
\end{align*}
\end{theorem}
As $\gamma\to\infty$, corresponding to progressively lighter tails, the tail dependent degradation in our bounds vanishes. 
In this limit, the total variation rate simplifies to $n^{-\beta/(2\beta+d)}$, matching the sub-Gaussian case studied in~\cite{zhang2024minimax} and the classical minimax rate for nonparametric density estimation~\cite{tsybakov2008nonparametric}.}

\paragraph{Alternative low-noise strategy.}
For completeness, we also explore in Appendix~\ref{sec:decoupled_score} an alternative approach that avoids early stopping by using a decoupled kernel score estimator in the low-noise regime where the diffusion time $t$ is small.
This variant comes with slightly stronger assumptions and a slower rate than Theorem~\ref{thm:poly-sampling}, so we present it as a complementary perspective rather than our main guarantee.

\subsection{Exponential Decay}\label{sec:exp_decay}

In this section, we study the convergence of the score estimator and the resulting sampling error for diffusion models when the target distribution $p_0$ has exponentially decaying tails. 
The main distinction from the polynomial tail case is the tail behavior of $p_0$. 
In parallel with the polynomial tail regime, we first establish the pointwise mean squared error bound for the score estimator under the exponential tail assumption.

\begin{theorem}
\label{thm:mse_exp}
Suppose Assumption 1 holds. For any $t \in (0, T]$, the proposed estimator $\hat{s}_t$ satisfies:
\begin{align*}
    \mathbb{E}\left[\int_{\mathbb{R}^d} \|\hat{s}_t(x) - s_t(x)\|^2 p_t(x) \rmd x\right] \leqslant \operatorname{polylog}(n)  n^{-1} t^{-1 - d/2}.
\end{align*}
\end{theorem}

Comparing this with Theorem 1, the most striking difference is the dependence on the sample size $n$. While the polynomial tail regime yields a rate $n^{-\frac{\gamma+1}{d+\gamma+1}}$ that explicitly degrades as tails become heavier (smaller $\gamma$), the exponential regime recovers the near-parametric rate $n^{-1}$ up to logarithmic factors regardless of the value of $\gamma$. This suggests that the standard diffusion model pipeline is statistically robust to a wide range of exponential tail behaviors.
The following theorem provides an upper bound on the cumulative score estimation error.
\begin{theorem}
\label{thm:exp_score}
Suppose that $p_0$ satisfies Assumption~\ref{asm:heavytail} and~\ref{asm:p0smooth}, for any $T>0$ and any $0<t_0<T$, we have
    \begin{align*}
        \int_{t_0}^T\E\left[\int_{\R^d}\|\hat s_t(x)-s_t(x)\|^2p_t(x)\,\rmd x\right]\,\rmd t\lesssim \operatorname{polylog}(n)n^{-1}t_0^{-d/2}\,.
    \end{align*}    
\end{theorem}
Under Assumption~\ref{asm:heavytail}, the target $p_0$ has stretched exponential tails, which includes the sub-Gaussian case as a special instance but also allows substantially more general decay profiles through the parameter $\gamma$. Theorem~\ref{thm:exp_score} shows that this additional generality does not deteriorate the statistical rate, as the cumulative score error over $[t_0,T]$ still scales as $\operatorname{polylog}(n) n^{-1}$ with an explicit small noise dependence $t_0^{-d/2}$. 
In particular, this bound recovers the same rate in the sub-Gaussian case~\cite[Corollary~3.7]{zhang2024minimax}.

The bound in Theorem~\ref{thm:exp_score} worsens as $t_0$ goes to zero, reflecting the ill-conditioning of the low-noise regime. 
We therefore stop the reverse dynamics at time $t_0$ and target $p_{t_0}=p_0*\varphi_{t_0}$, which incurs an {early stopping bias} $\operatorname{TV}(p_0,p_{t_0})$. 
The next theorem controls this bias under exponential tails.
\begin{theorem}
\label{thm:early_stopping_exp}
Suppose that $p_0$ satisfies Assumption~\ref{asm:heavytail} and Assumption~\ref{asm:p0smooth} with $\beta \in (0, 2]$. Then, the early stopping bias satisfies
\begin{align*}
\operatorname{TV}(p_0, p_{t_0}) \lesssim t_0^{\beta/2} \cdot (\log(1/t_0))^{\frac{d}{2\gamma}}\,.
\end{align*}
In particular, up to polylogarithmic factors, the rate is dominated by $t_0^{\beta/2}$.
\end{theorem}
Building on this result and the score estimation guarantees in Theorem~\ref{thm:exp_score}, we now establish a bound on the total variation distance between the target distribution $p_0$ and the distribution of samples produced by the reverse time diffusion dynamics~\eqref{eq:backward1}.
\begin{theorem}
\label{thm:exp_sampling}
Suppose that Assumptions~\ref{asm:heavytail} and~\ref{asm:p0smooth} hold. 
Set $
t_0 = n^{-\frac{2}{2\beta+d}}$ and $
T = n^{\frac{2\beta}{2\beta+d}}.$
Then, there exists a score estimator $(\hat s_t)_{t\in[0,T]}$ such that the sample
$\hat Y_{T-t_0}$ from the corresponding reverse time diffusion process~\eqref{eq:backward1} satisfies
\begin{align*}
\E\bigl[\operatorname{TV}\bigl(p_0,\law(\hat Y_{T-t_0})\bigr)\bigr]
\;\lesssim\;
\operatorname{polylog}(n)\,n^{-\frac{\beta}{2\beta+d}}.
\end{align*}
\end{theorem}

The choices of $t_0$ and $T$ align with prior analyses for sub-Gaussian targets and lead to the same rate $n^{-\beta/(2\beta+d)}$~\cite[Theorem 3.8]{zhang2024minimax}.
Allowing $p_0$ to have stretched exponential tails does not worsen the convergence rate compared with the sub-Gaussian or compact support setting. This indicates that the conventional diffusion model is robust throughout the exponentially decaying regime.

Comparing our analysis in the polynomial and exponential regimes reveals a qualitative dichotomy. 
In the polynomial case, the tail index $\gamma$ appears explicitly through the factor $n^{-(\gamma+1)/(d+\gamma+1)}$ in Theorem~\ref{thm:t0scoreerror}.
In contrast, in the exponential case, the rate improves to $\bigo(n^{-1})$ up to logarithmic factors in Theorem~\ref{thm:exp_score}, and in the corresponding sampling guarantee (Theorem~\ref{thm:exp_sampling}) the dependence on $\gamma$ effectively disappears from the leading term.
The intuition comes from the geometry of where most of the probability mass lives. 
For polynomially decaying tails, capturing most of the mass requires enlarging the region of interest at a \emph{polynomial} rate in $n$.
Consequently, the size of the domain that must be covered grows rapidly,
and the available samples are spread thinly over a large domain. 
For exponentially decaying tails, the region containing almost all the mass expands only \emph{logarithmically} with $n$, so its volume grows much more slowly. 
In this case, the estimation error scales like $n^{-1}$ up to polylogarithmic factors, and the tail index $\gamma$ only enters through these logarithmic terms, which explains its apparent disappearance from the main convergence rate.

\section{Main Results: Minimax Lower Bound for Score Estimation}\label{sec:minimax_lb}

To complement our upper bounds and demonstrate the rate optimality of the proposed score estimator in the heavy-tailed regime, we establish the matching minimax lower bounds for the score estimation error. The main technical challenge is to construct a family of hypotheses (candidate data distributions) that are statistically difficult to distinguish while inducing well-separated score functions. 
Below, we first present the minimax lower bounds in both the polynomial decay and exponential decay regimes.
Then, we describe the least favorable family that drives the construction and serves as the main technical ingredient. 
Since the proofs for both tail regimes follow the same overall strategy and the polynomial decay case is more technically involved, we focus on that setting.
The complete proof is deferred to Appendix~\ref{app:prooflowerbound}.

\subsection{Minimax Lower Bound  under Polynomial Tail Regime}
In this part, we show that the MSE bound obtained in Theorem~\ref{thm:MSEst} is nearly minimax optimal.

\begin{theorem}\label{thm:lower_bound_poly}
    Let $\mathcal{H}_\gamma$ denote the class of heavy-tailed distributions satisfying Assumption~\ref{asm:p0tail} with tail index $\gamma$. For any large $n$ and all diffusion time $t$ satisfying $n^{-\frac{2(d+2\gamma+2)}{d(d + 2(\gamma + 1)) + 2\beta(d+\gamma+1)}}\lesssim t\lesssim 1$, the minimax risk satisfies
    \begin{align*}
        \inf_{\hat s_t}\sup_{p_0 \in \mathcal{H}_\gamma\cap\mathcal P_{\mathcal S}(\beta, L)} \E\left[ \int_{\R^d}\|\hat s_t(x) - s_t(x)\|^2p_t(x)\rmd x\right] \gtrsim n^{-\frac{\gamma+1}{d+\gamma+1}-\delta}t^{-1-\frac{\gamma+1}{d+\gamma+1}+\delta},
    \end{align*}
    up to logarithmic factors, for any arbitrarily small $\delta>0$.
\end{theorem}
This lower bound matches the rate in Theorem~\ref{thm:MSEst} up to logarithmic factors, establishing its minimax optimality. 
In particular, the dependence on the tail index 
$\gamma$ is intrinsic and cannot be improved uniformly over the class.

{ 
\subsection{Minimax Lower Bound  under Exponential Tail Regime}
In this part, we establish that the MSE rate in Theorem~\ref{thm:mse_exp} is nearly minimax optimal.
\begin{theorem}
\label{thm:lower_bound_exp}
Let $\mathcal{H}'_\gamma$ denote the class of heavy-tailed distributions satisfying Assumption~\ref{asm:heavytail} with tail index $\gamma$. 
For all $n$ large enough and all diffusion times $t\gtrsim n^{-\frac{2}{d+2\beta}}$, the minimax risk satisfies
 \begin{align*}
        \inf_{\hat s_t}\sup_{p_0 \in \mathcal{H}'_\gamma\cap\mathcal P_{\mathcal S}(\beta, L)} \E\left[ \int_{\R^d}\|\hat s_t(x) - s_t(x)\|^2p_t(x)\rmd x\right] \gtrsim  n^{-1}\, t^{-1-d/2}.
    \end{align*}
\end{theorem}
This lower bound matches the rate in Theorem~\ref{thm:mse_exp} up to logarithmic factors and therefore establishes minimax optimality.
}

\subsection{Construction of Least Favorable Distributions}

The proof for the lower bounds uses Fano's method applied to a carefully designed finite subset of target distributions. In contrast to classical nonparametric lower bounds, which often perturb densities on a fixed compact domain such as \([-1,1]^d\), our setting requires perturbations placed in the tail region to reflect polynomial decay.

\vspace{0.2cm}
\noindent\textbf{Perturbation domain in the tail.}
We introduce localized perturbations on a hypercube grid contained in $[R,2R]^d$, where $R>0$ is a large radius that will be chosen as a function of the sample size $n$ and the time $t$. This choice is essential for two reasons:
\begin{itemize}
    \item In the tail region, the baseline density is small (on the order of $\|x\|^{-(d+\gamma+1)}$), which keeps the KL divergence between hypotheses small and therefore makes them hard to distinguish.
    \item At the same time, we need the corresponding score functions to remain sufficiently separated. The achievable separation is controlled by the trade-off between the decaying baseline density and the magnitude of the perturbation, which depends on the tail index $\gamma$.
\end{itemize}

\noindent\textbf{Perturbed hypotheses.}
Let \(q_0\) be a reference heavy-tailed density of the form
\(q_0(x)\propto (1+\|x\|^2)^{-(d+\gamma+1)/2}\).
Fix a grid \(\{x_{\mathbf i}\}_{\mathbf i\in\mathcal I}\subset [R,2R]^d\) with spacing \(h\), where \(\mathcal I\) denotes the corresponding index set.
For a binary vector \(b=(b_{\mathbf i})_{\mathbf i\in\mathcal I}\in\{0,1\}^{|\mathcal I|}\), define the perturbed density
\begin{align*}
q_b(x)
= q_0(x) + \epsilon \sum_{\mathbf i\in\mathcal I} b_{\mathbf i}\,\omega(x-x_{\mathbf i}),
\end{align*}
where \(\omega\) is a smooth bump function supported at scale \(h\).
We then select a finite subset $\mathcal B\subseteq \{0,1\}^{|\mathcal I|}$ and consider the hypothesis family $\{q_b\}_{b\in\mathcal B}$.

\vspace{0.2cm}
\noindent\textbf{Coupling of parameters.}
The lower bound follows from a careful balance between the perturbation amplitude \(\epsilon\), the tail location \(R\), and the grid scale \(h\).
As \(n\to\infty\), the perturbations must be placed further into the tail (so \(R\to\infty\)) to exploit the small baseline density, while \(\epsilon\) must shrink to satisfy the information theoretic constraints.
In particular, we choose
\[
R \asymp n^{\frac{1}{d+\gamma+1}}\,t^{-\frac{d}{2(d+\gamma+1)}},
\]
which couples the sample size with the geometry induced by the polynomial tail.

\section{Proof Sketch}
\label{sec:sketch}
In this section, we outline high-level proof sketches for Theorems~\ref{thm:poly-sampling} and~\ref{thm:exp_sampling}.

\subsection{Proof Sketch for the Sampling Error in the Polynomial Decay Regime (Theorem~\ref{thm:poly-sampling})}
The proof unifies our results on score estimation and early stopping to derive the final sampling guarantee. 
We decompose the total variation error between the data distribution $p_0$ and the generated distribution $\law(\hat{Y}_{T-t_0})$ using the triangle inequality. First,
\[
\operatorname{TV}\bigl(p_0,\law(\hat Y_{T-t_0})\bigr)
\leqslant 
\operatorname{TV}(p_0,p_{t_0})
+
\operatorname{TV}\bigl(p_{t_0},\law(\hat Y_{T-t_0})\bigr),
\]
which separates the early stopping bias and the reverse process error. 
Next, let $\tilde Y$ denote the backward process with the \emph{same} learned score but initialized from $p_T$. Then,
\[
\operatorname{TV}\bigl(p_{t_0},\law(\hat Y_{T-t_0})\bigr)
\leqslant
\operatorname{TV}\bigl(p_{t_0},\law(\tilde Y_{T-t_0})\bigr)
+
\operatorname{TV}\bigl(\law(\tilde Y_{T-t_0}),\law(\hat Y_{T-t_0})\bigr).
\]
The first term corresponds to the score estimation error. For the second term, TV contraction under the backward kernel yields
\[
\operatorname{TV}\bigl(\law(\tilde Y_{T-t_0}),\law(\hat Y_{T-t_0})\bigr)
\leqslant
\operatorname{TV}\bigl(p_T,\mathcal N(0,TI_d)\bigr),
\]
which we interpret as the initialization error. Combining these bounds gives the following decomposition.
\begin{align*}
\operatorname{TV}(p_0,\law(\hat{Y}_{T-t_0})) \leqslant \underbrace{\operatorname{TV}(p_0, p_{t_0})}_{\text{Early Stopping Bias}} + \underbrace{\operatorname{TV}(p_{t_0}, \law(\tilde{Y}_{T-t_0}))}_{\text{Score Matching Error}} + \underbrace{\operatorname{TV}(p_T, \mathcal{N}(0, TI_d))}_{\text{Initialization Error}}.
\end{align*}

\noindent\textbf{Score Matching Error (Theorems~\ref{thm:MSEst} and~\ref{thm:t0scoreerror}).} By Pinsker's inequality and Girsanov's theorem, the sampling error is controlled by the square root of the cumulative score estimation error. 
\[
\operatorname{TV}\bigl(\law(\tilde Y_{T-t_0}),p_{t_0}\bigr)
\;\lesssim\;
\biggl(
\int_{t_0}^T
\E\Bigl[\!\int_{\R^d}\|\hat s_t(x)-s_t(x)\|^2 p_t(x)\,\rmd x\Bigr] \rmd t
\biggr)^{1/2}.
\]
Combining this with the heavy-tailed rate in Theorem~\ref{thm:t0scoreerror}, we obtain
\[
\operatorname{TV}\bigl(\law(\tilde Y_{T-t_0}),p_{t_0}\bigr)
\;\lesssim\;
\operatorname{polylog}(n)\,
n^{-\frac{\gamma+1}{2(d+\gamma+1)}}\,
t_0^{-\frac{d(\gamma+1)}{4(d+\gamma+1)}}.
\]
This term blows up as $t_0 \to 0$, which favors choosing a larger cutoff $t_0$ to avoid the small noise regime.

\noindent\textbf{Early stopping bias (Theorem~\ref{thm:early_stopping}).}
To avoid the ill-conditioned small noise regime near $t=0$, we stop the reverse process at $t_0$ and target the smoothed distribution $p_{t_0} = p_0 * \varphi_{t_0}$ instead of $p_0$.
Theorem~\ref{thm:early_stopping} shows that this induces a total variation bias
\[
\operatorname{TV}(p_0,p_{t_0}) \lesssim t_0^{\frac{\beta(\gamma+1)}{d+2(\gamma+1)+2\beta}},
\]
which vanishes as $t_0 \to 0$ at a rate governed by the smoothness $\beta$ and tail index $\gamma$.
This term therefore favors choosing a smaller $t_0$.

\noindent\textbf{Initialization error (Lemma~\ref{lem:TV}).}
The discrepancy between the latent distribution $p_T = p_0 * \varphi_T$ and the Gaussian prior $\mathcal{N}(0,TI_d)$ is controlled by Lemma~\ref{lem:TV}, which yields
\[
\operatorname{TV}(p_T,\varphi_T) \lesssim T^{-1/2}.
\]
Thus, the initialization error can be made negligible by choosing $T$ sufficiently large.

The final minimax rate is achieved by balancing the trade-off between the early stopping bias and the score matching error. Equating these two rates yields the optimal stopping time $t_0$ and the stated convergence rate.

Below, we present proof sketches for Theorems~\ref{thm:MSEst} and~\ref{thm:early_stopping}, explain the main technical challenges arising in the heavy-tailed regime, and highlight how our proof techniques differ from existing analyses developed for light-tailed settings.

\subsubsection{Proof Sketch of Theorem~\ref{thm:MSEst}}
\label{sec:proof_poly_MSEst}
Our main approach is to decompose the error using a thresholding argument.
We control the weighted mean integrated squared error (MISE) of the score estimator $\hat s_t(x)$ by splitting the domain according to (i) the magnitude of the true density and (ii) the thresholding applied to the estimator.

First, we decompose according to the true density $p_t$. We partition $\R^d$ into a high-density bulk region
\[
\mathcal G_1 := \{x \in \R^d : p_t(x) > c_\alpha \rho_n\}
\]
and a low-density tail region
\[
\mathcal G_2 := \{x \in \R^d : p_t(x) \leqslant c_\alpha \rho_n\},
\]
where $\rho_n$ is the given truncation threshold and $c_\alpha>0$ is a constant to be specified. 
Second, we exploit the thresholding built into the estimator: by construction, $\hat s_t(x)$ is set to zero whenever the estimated density falls below the threshold, i.e.\ when $\hat p_t(x) < \rho_n$.

These two decompositions are coupled through the relationship between $p_t$ and $\hat p_t$, which is controlled by a uniform concentration bound.

\begin{lemma}\label{lem:Gausspthat}
For any fixed $\alpha>0$, there exists a constant $C_\alpha>0$, depending only on $p_0$ and the dimension $d$, such that with probability at least $1-n^{-\alpha}$,
\begin{align*}
    \bigl|\hat p_t(x)-p_t(x)\bigr|
    < C_{\alpha}\!\left(
      \frac{\log n}{n(2\pi t)^{d/2}}
      +\sqrt{\frac{p_t(x)\log n}{n(2\pi t)^{d/2}}}
    \right),
    \qquad \text{for all } x\in\mathbb{R}^d,
\end{align*}
where $C_\alpha=\max\{\sqrt{8\alpha},\,16\alpha/3\}$.
\end{lemma}

Let $A_\alpha$ denote the high-probability event on which Lemma~\ref{lem:Gausspthat} holds.
Using this event together with the bulk/tail partition, we decompose the error as
\begin{align*}
    \E\Bigl[\int_{\R^d}\|\hat s_t(x)-s_t(x)\|^2p_t(x)\,\rmd x\Bigr]
    &=\underbrace{\int_{\mathcal G_1}\E[\|\hat s_t(x)-s_t(x)\|^2\,\mathbbm 1_{A_\alpha}]\,p_t(x)\,\rmd x}_{\text{I}} \\
     & \qquad +\underbrace{\int_{\mathcal G_2}\E[\|\hat s_t(x)-s_t(x)\|^2\,\mathbbm 1_{A_\alpha}]\,p_t(x)\,\rmd x}_{\text{II}}\\
    &\qquad+\underbrace{\int_{\R^d}\E[\|\hat s_t(x)-s_t(x)\|^2\,\mathbbm 1_{A_\alpha^c}]\,p_t(x)\,\rmd x}_{\text{III}}.
\end{align*}
Within terms~\text{I} and~\text{II}, we further split according to whether $\hat p_t(x)\geqslant \rho_n$ (active region, where the score is estimated) or $\hat p_t(x)<\rho_n$ (truncated region, where the score is set to zero).

\medskip
\noindent\textbf{Step 1: Lower bounds on the denominator.}
The main technical difficulty comes from the unstable denominator in
\[
\hat s_t(x) = \frac{\nabla \hat p_t(x)}{\hat p_t(x)}\,\mathbbm 1_{\{\hat p_t(x)\geqslant \rho_n\}}.
\]
We first derive lower bounds on $\hat p_t(x)$ in the different regions:
on the intersection $\mathcal G_1 \cap A_\alpha$, Lemma~\ref{lem:Gausspthat} implies $\hat p_t(x)\asymp p_t(x)$ and, in particular, $\hat p_t(x)\geqslant \frac12 p_t(x)$.  
In the remaining active regions, we use the explicit threshold $\hat p_t(x)\geqslant \rho_n$.

\medskip
\noindent\textbf{Step 2: Reduction to pointwise MSE bounds.}
Once the denominator is controlled from above, the score error can be reduced to linear combinations of the errors of $\hat p_t$ and $\nabla\hat p_t$.
These are bounded by the following pointwise MSE estimates, which crucially scale with the true density $p_t(x)$.

\begin{proposition}[MSE for $\hat p_t$]\label{prop:GaussMSEpt}
The mean squared error of $\hat p_t(x)$ satisfies
\[
    \mathrm{MSE}(\hat p_t(x))
    \;\leqslant\; \frac{p_t(x)}{n(2\pi t)^{d/2}}\,.
\]
\end{proposition}

\begin{proposition}[MSE for $\nabla\hat p_t$]\label{prop:GaussMSEhatpt}
The mean squared error of $\nabla\hat p_t(x)$ satisfies
\[
    \mathrm{MSE}(\nabla\hat p_t(x))
    \;\leqslant\; \frac{p_t(x)}{n t (2\pi t)^{d/2}}\,.
\]
\end{proposition}

\medskip
\noindent\textbf{Step 3: Integration using bulk and tail lemmas.}
Substituting these MSE bounds into terms~\text{I} and~\text{II} leads to integrals over $\mathcal G_1$ and $\mathcal G_2$ with weights involving $p_t(x)$ and $s_t(x)$. 
These integrals are controlled using the next two lemmas, which quantify the geometry of the bulk and tail regions under the heavy-tailed assumption.

\begin{lemma}[Also Lemma~\ref{lem:G1_v2}]\label{lem:G1}
Under Assumption~\ref{asm:p0tail}, we have
\begin{align*}
    |\mathcal G_1|
    &\lesssim \bigl(C_\gamma(t)\rho_n\bigr)^{-\frac{d}{1+\gamma+d}},\\
    \int_{\mathcal G_1}\|s_t(x)\|^2\,\rmd x
    &\lesssim \big(C_\gamma(t)\big)^{-\frac{1}{m}}\,m\,t^{-1}\rho_n^{-\frac{1}{m}-\frac{d}{(1+\gamma+d)m'}}.
\end{align*}
\end{lemma}

\begin{lemma}[Also Lemma~\ref{lem:G2st_v2}]\label{lem:G2st}
Under Assumption~\ref{asm:p0tail}, it holds that
\begin{align*}
    \int_{\mathcal G_2}p_t(x)\,\rmd x
    &\lesssim \big(C_\gamma(t)\big)^{\frac{d}{d+\gamma+1}}\rho_n^{\frac{\gamma+1}{d+\gamma+1}},\\
    \int_{\mathcal G_2}\|s_t(x)\|^2p_t(x)\,\rmd x
    &\lesssim \varepsilon^{-1}t^{-1}\big(C_\gamma(t)\big)^{\frac{d(1-\varepsilon)}{d+\gamma+1}}\rho_n^{\frac{(\gamma+1)(1-\varepsilon)}{d+\gamma+1}}.
\end{align*}
\end{lemma}

Here $m,m'>1$ satisfying $1/m+1/m'=1$, $C_\gamma(t)$ denotes a function of $\gamma, t$, and $d$, whose explicit form is given in the appendix. 
Combining the bounds for terms~\text{I}–\text{III}, and optimally balancing the contributions from $\mathcal G_1$ and $\mathcal G_2$ via the choice of $\rho_n$, yields the claimed rate in Theorem~\ref{thm:MSEst}.

\paragraph{Remark 1 (Statistical bottlenecks in the heavy-tailed regime).}
We note that the rate $n^{-\frac{\gamma+1}{d+\gamma+1}}$ in Theorem~\ref{thm:MSEst} is not an artifact of a loose analysis. 
It reflects two intrinsic geometric effects of heavy tails that fundamentally change the bias–variance trade-off compared to the traditional sub-Gaussian case:
\begin{itemize}
\item \textbf{Polynomial volume expansion in the bulk ($\mathcal G_1$).}
In light-tailed settings, the high-density region $\mathcal G_1$ (where estimation is stable) typically has volume that grows only logarithmically with $n$. 
Under heavy tails, Lemma~\ref{lem:G1} shows that $|\mathcal G_1|$ grows polynomially as $\rho_n^{-d/(d+\gamma+1)}$. 
This rapid expansion dilutes the effective sample density and inflates the integrated variance in the bulk.

\item \textbf{Substantial probability mass in the tail ($\mathcal G_2$).}
The truncation error in $\mathcal G_2$ is controlled by the probability mass below the threshold $\rho_n$. 
For sub-Gaussian tails, this mass decays exponentially, allowing aggressive truncation with negligible error. 
In contrast, Lemma~\ref{lem:G2st} shows that polynomial tails retain non-negligible mass in $\mathcal G_2$, scaling as $\rho_n^{(\gamma+1)/(d+\gamma+1)}$, which forces a more conservative choice of $\rho_n$ and creates a tight trade-off between variance in the bulk and bias from the tail.
\end{itemize}

\paragraph{Remark 2 (Difficulties of score estimation in the low-noise regime).}
Our analysis also clarifies why the naive kernel estimator becomes suboptimal as $t \to 0$.
In unbounded settings, both exponentially decaying and polynomially decaying targets generate far tail regions where data points are extremely sparse.
As $t$ decreases, the effective kernel scale $\sqrt{t}$ becomes smaller than the typical distance between samples in these regions, so the estimator receives almost no data, and its local variance blows up.
This phenomenon does not arise under compact support, where the main challenge is boundary bias rather than large empty regions.

The key distinction is how much probability mass these sparse regions carry.
For exponential-type decay, the mass in the far tail shrinks extremely fast, so the contribution of high variance regions to the $L^2$ error stays negligible and the convergence rate can match the compact-support case (see, for example,~\cite{dou2024optimal}).
For polynomial decay, in contrast, the tail of $p_0$ decreases much more slowly, so these sparse regions still contain substantial probability mass.
The density $p_t$ does not drop quickly enough to compensate for the variance blow-up, and these regions end up dominating the global estimation error and degrading the convergence rate.

\subsubsection{Proof Sketch of Theorem~\ref{thm:early_stopping}}
The proof starts by splitting the total variation distance into a contribution from a central region of radius $R$ and a complementary tail region. 
In the central region, we bound the $L^1$-error by the $L^2$-error via Hölder’s inequality, which introduces a factor proportional to the square root of the volume, i.e.\ $R^{d/2}$. 
The $L^2$ smoothing error $\|p_0 - p_{t_0}\|_{L^2}$ is then controlled in the frequency domain using Plancherel's theorem: exploiting the Sobolev regularity $\beta$ of $p_0$, we obtain a decay of order $t_0^{\beta/2}$. 
In the tail region, the error is controlled by the heavy-tail assumption, which implies a polynomial decay of the mass outside the ball of radius $R$, of order $R^{-(\gamma+1)}$. 
Putting these estimates together yields an upper bound of the form
\[
R^{d/2} t_0^{\beta/2} \,+\, R^{-(\gamma+1)}.
\]
The claimed rate follows by choosing $R$ to balance the smoothing error in the central region against the truncation error in the tail.

\subsection{Proof Sketch for the Sampling Error in the Exponential Decay Regime (Theorem~\ref{thm:exp_sampling})}
The proof follows the same step as in the polynomial decay regime, but the geometric control underlying the score estimation bounds is substantially sharper under exponential tails. 
As a result, the score matching term no longer suffers the polynomial heavy-tail penalty.

We start from the same decomposition of the total variation error
\begin{align*}
\operatorname{TV}(p_0,\law(\hat{Y}_{T-t_0})) \leqslant \underbrace{\operatorname{TV}(p_0, p_{t_0})}_{\text{Early Stopping Bias}} + \underbrace{\operatorname{TV}(p_{t_0}, \law(\tilde{Y}_{T-t_0}))}_{\text{Score Matching Error}} + \underbrace{\operatorname{TV}(p_T, \mathcal{N}(0, TI_d))}_{\text{Initialization Error}}.
\end{align*}
\noindent\textbf{Score Matching Error (Theorems~\ref{thm:mse_exp} and~\ref{thm:exp_score}).}
As before, Pinsker's inequality and Girsanov's theorem control the sampling error by the square root of the cumulative score mean squared error. The key distinction is that, in the exponential regime, the relevant geometric quantities are much better behaved: the effective support volume $|\G_1|$ grows only logarithmically (Lemma~\ref{lem:G1_expdecay}) and the tail mass decays super-polynomially (Lemma~\ref{lem:tail_prob_exp}). Substituting these bounds into the MSE bound for score estimation yields the faster cumulative rate
\[
\operatorname{TV}\bigl(\law(\tilde Y_{T-t_0}),p_{t_0}\bigr)
\;\lesssim\;
\operatorname{polylog}(n)\,
n^{-1/2}\,
t_0^{-d/4}.
\]
In particular, the dependence on $n$ is near parametric, reflecting the robustness of the diffusion pipeline under exponential tails.

\noindent\textbf{Early stopping bias (Theorem~\ref{thm:early_stopping_exp}).} Stopping the reverse dynamics at time $t_0$ again replaces $p_0$ by the smoothed target $p_{t_0}=p_0*\varphi_{t_0}$. 
In the exponential regime, the truncation component is negligible, and the bias is governed solely by Sobolev regularity via 
\[
\operatorname{TV}(p_0,p_{t_0}) \lesssim t_0^{\beta/2}.
\]
Thus, as in the polynomial case, this term favors taking $t_0$ small.

\noindent\textbf{Initialization error (Lemma~\ref{lem:TV}).} The initialization mismatch between $p_T=p_0*\varphi_T$ and the Gaussian prior satisfies the same bound as before,
\[
\operatorname{TV}(p_T,\varphi_T)\;\lesssim\;T^{-1/2},
\]
and can be made negligible by choosing $T$ sufficiently large.

The final rate follows by balancing the early stopping bias $t_0^{\beta/2}$ with the score matching error $n^{-1/2}t_0^{-d/4}$. This yields the optimal cutoff
\[
t_0 \asymp n^{-\frac{2}{2\beta+d}},
\]
and the resulting convergence rate
\[
\operatorname{TV}\bigl(p_0,\law(\hat{Y}_{T-t_0})\bigr)
\;\lesssim\;
\operatorname{polylog}(n)\, n^{-\frac{\beta}{2\beta+d}}\,.
\]

\section{Discussion}

In this work, we characterize the fundamental minimax limits of score estimation and the induced sampling error in the canonical SGM framework for heavy-tailed target distributions.
We show that the standard Gaussian noising pipeline is statistically robust to tail behavior when the target density has exponential decay. 
In contrast, for polynomially decaying targets, this robustness breaks down, and designing principled remedies remains an open problem.

One possible direction is to modify the forward process by adopting heavy-tailed noise, such as stable or Lévy-type perturbations~\cite{yoon2023score,shariatianheavy}. 
However, this moves beyond the Gaussian diffusion family and introduces substantially different analytic challenges, so extending minimax theory to this regime is nontrivial and beyond the scope of the present work. Another promising approach is to consider tempered or data-dependent noise schedules that avoid placing excessive mass in regions that dominate the statistical error, and we leave a systematic investigation of such remedies for future work.

Furthermore, in the polynomial decay regime, we establish minimax optimality for score estimation and provide guarantees for the sampling error of the continuous time reverse diffusion driven by the learned score, but it remains unclear whether this sampling rate is minimax optimal. 
Closing this gap would require either matching minimax lower bounds for the sampling problem or sharper analyses that more precisely characterize how score estimation error propagates through the reverse diffusion dynamics. We leave this question for future work.

\section{Acknowledgements}
This work is supported by the City University of Hong Kong Startup Fund and Hong Kong RGC Grant 21306325.

\bibliographystyle{amsalpha}
{\small \bibliography{bib}}
\newpage
\appendix

\clearpage
\section{Appendix}

\paragraph{Additional Notation.}
{We denote by $\mathbb{B}_r(x) := \{y \in \mathbb{R}^d : \|y-x\| \leqslant r\}$ the Euclidean ball of radius $r$ centered at $x$, and simply write $\mathbb{B}_r$ for $\mathbb{B}_r(0)$.}
\subsection{Proof of Section~\ref{sec:kde}}

\begin{proof}[Proof of Proposition~\ref{prop:GaussMSEpt}]
By the definition of MSE and the construction of the $\hat p_t,$ we have
\begin{align*}
    {\rm{MSE}}(\hat p_t(x))\leqslant \text{Var}(\hat p_t(x))&\leqslant \dfrac{1}{n}\text{Var}(\varphi_t(X^{(i)}-x))\,.
\end{align*}
Let $Z_0\sim p_0$ be an independent copy of $\{X^{(i)}\}_{i=1}^n.$ 
We then have
\begin{align*}
    \dfrac{1}{n}\text{Var}(\varphi_t(X^{(i)}-x))&\leqslant \dfrac{1}{n}\E[\varphi_t^2(Z_0-x)]\\
    &= \dfrac{1}{n}\int_{\R^d}\dfrac{1}{(2\pi t)^d}\exp\left(-\dfrac{\|x-y\|_2^2}{t}\right)p_0(y)\,\rmd y\\
    &\leqslant \dfrac{1}{n(2\pi t)^{d/2}}\int_{\R^d}\dfrac{1}{(2\pi t)^{d/2}}\exp\left(-\dfrac{\|x-y\|_2^2}{2t}\right)p_0(y)\,\rmd y\\
    &=\dfrac{p_t(x)}{n(2\pi t)^{d/2}}\,.
\end{align*}
Combining this with previous display gives the desired result.
\end{proof}

\begin{proof}[Proof of Proposition~\ref{prop:GaussMSEhatpt}]
Let $Z_0\sim p_0$ be an independent copy of $\{X^{(i)}\}_{i=1}^n.$ 
    Since $\nabla\hat p_t(x)$ is unbiased,  we have
    \begin{align*}
        \text{MSE}(\nabla\hat p_t(x))
        &=\text{Var}(\nabla\hat p_t(x))\\
        &=\dfrac{1}{nt^2}\text{Var}((X^{(i)}-x)\varphi_t(X^{(i)}-x))\\
        &\leqslant \dfrac{1}{nt^2}\E\left[(Z_0-x)^2\varphi_t^2(Z_0-x)\right]\\
        &=\dfrac{1}{nt^2}\int_{\R^d}\dfrac{(y-x)^2}{(2\pi t)^d}\exp\left(-\dfrac{\|x-y\|_2^2}{t}\right)p_0(y)\rmd y\\
        &\leqslant \dfrac{1}{nt(2\pi t)^{d/2}}\int_{\R^d}\dfrac{1}{(2\pi t)^{d/2}}\exp\left(-\dfrac{\|x-y\|_2^2}{2t}\right)p_0(y)\rmd y\\
        &=\dfrac{p_t(x)}{nt(2\pi t)^{d/2}}
    \end{align*}
    as desired.
\end{proof}

\subsection{Proof of Section~\ref{sec:poly}}
Define the event
\begin{align*}
    A_\alpha=\left\{x\in\R^d:|\hat p_t(x)-p_t(x)|<C_\alpha\left(\dfrac{\log n}{n(2\pi t)^{d/2}}+\sqrt{\dfrac{p_t(x)\log n}{n(2\pi t)^{d/2}}}\right)\right\}\,.
\end{align*}
By Lemma~\ref{lem:Gausspthat}, it holds that $\P(A_\alpha^c)<n^{-\alpha}$. 
We further decompose the space into the bulk region $\G_1$ and the tail region $\G_2$ through
\begin{align*}
    \G_1:=\{x: p_t(x)>c_\alpha\rho_n\},
    \qquad
    \G_2:=\{x: p_t(x)\leqslant c_\alpha\rho_n\},
\end{align*}
where $c_\alpha$ is chosen to satisfy $c_\alpha>2$ and $c_\alpha>2C_\alpha(1+\sqrt{c_\alpha})$.
Then for any $x\in \G_1\cap A_\alpha$,
\begin{align*}
    \hat p_t(x)
    &> p_t(x)-C_\alpha\bigl(\rho_n+\sqrt{p_t(x)\rho_n}\bigr)\\
    &> p_t(x)-C_\alpha\left(\frac{p_t(x)}{c_\alpha}+\frac{p_t(x)}{\sqrt{c_\alpha}}\right)
     > \Bigl(1-C_\alpha(c_\alpha^{-1}+c_\alpha^{-1/2})\Bigr)p_t(x)
     > \frac{1}{2}p_t(x),
\end{align*}
and consequently $\hat p_t(x)>\frac{1}{2}p_t(x)> \frac{c_\alpha}{2}\rho_n>\rho_n$.
This shows that in the bulk region $\G_1$, the density estimate is bounded away
from zero (on event $A_\alpha$), whereas $\G_2$ isolates the low-density region
where thresholding is needed for stability.
\subsubsection{Proof of Theorem~\ref{thm:MSEst}}\label{sec:proofthm1}
\begin{proof}
The proof relies on partitioning the integration domain into a high-density region $\G_1$, where the kernel estimator concentrates well, and a low-density tail region $\G_2$, where the estimator is truncated. We divide the $L^2$ score matching error into three parts
\begin{align*}
    &\quad\E\left[\int_{\R^d}\|\hat s_t(x)-s_t(x)\|^2p_t(x)\,\rmd x\right]\\
    &=\int_{\R^d}\E\left[\|\hat s_t(x)-s_t(x)\|^2\right]p_t(x)\,\rmd x\\
    &=\underbrace{\int_{\G_1}\E\left[\|\hat s_t(x)-s_t(x)\|^2\mathbbm 1\{\omega: x\in A_\alpha(\omega)\}\right]p_t(x)\,\rmd x}_{\text I}+\underbrace{\int_{\G_2}\E\left[\|\hat s_t(x)-s_t(x)\|^2\mathbbm 1\{\omega:x\in A_\alpha(\omega)\}\right]p_t(x)\,\rmd x}_{\text{II}}\\
    &\quad +\underbrace{\int_{\R^d}\E\left[\|\hat s_t(x)-s_t(x)\|^2\mathbbm 1\{\omega:x\in A_\alpha^c(\omega)\}\right]p_t(x)\,\rmd x}_{\text{III}}\,.
\end{align*}
\textbf{Step 1.} Note that when $x\in \G_1\cap A_\alpha$, $\hat p_t(x)>\max\{\rho_n,\frac{1}{2}p_t(x)\}$, then
\begin{align}\label{eq:term1}
    \text I&\leqslant \int_{\G_1}\E[\|\hat s_t(x)-s_t(x)\|^2\mathbbm 1\{\hat p_t(x)>\max(\rho_n,\frac{1}{2}p_t(x))\}]p_t(x)\,\rmd x\\
    &\leqslant \int_{\G_1}2\E\left[\dfrac{\|\nabla \hat p_t(x)-\nabla p_t(x)\|^2+\|s_t(x)\|^2|\hat p_t(x)-p_t(x)|^2}{\hat p_t(x)^2}\mathbbm 1\{\hat p_t(x)>\frac{1}{2}p_t(x)\}\right]p_t(x)\,\rmd x\\
    &\leqslant 8\int_{\G_1}\E\left[\|\nabla \hat p_t(x)-\nabla p_t(x)\|^2+\|s_t(x)\|^2|\hat p_t(x)-p_t(x)|^2\right]\frac{1}{p_t(x)}\,\rmd x\\
    &\leqslant 8\int_{\G_1}\left(\text{MSE}(\nabla \hat p_t(x))+\text{MSE}(\hat p_t(x))\cdot\|s_t(x)\|^2\right)\cdot \dfrac{1}{p_t(x)}\,\rmd x
\end{align}
By Proposition~\ref{prop:GaussMSEpt} and~\ref{prop:GaussMSEhatpt}, we have 
\begin{align}\label{eq:term1_1}
    \text I\leqslant 8|\G_1|\cdot \dfrac{1}{nt(2\pi t)^{d/2}}+8\int_{\G_1}\|s_t(x)\|^2\,\rmd x\cdot\dfrac{1}{n(2\pi t)^{d/2}}
\end{align}

\begin{lemma}\label{lem:G1_v2}
Under Assumption~\ref{asm:p0tail}, we have
    \begin{align*}
        |\G_1|\lesssim (C_\gamma(t)\rho_n)^{-\frac{d}{1+\gamma+d}}\,,
    \end{align*}
    and
    \begin{align*}
        \int_{\G_1}\|s_t(x)\|^2\rmd x\lesssim C_\gamma(t)^{-\frac{1}{m}}mt^{-1}\rho_n^{-\frac{1}{m}-\frac{d}{(1+\gamma+d)m'}}\,,
    \end{align*}
    where
    \begin{align*}
    C_\gamma(t)=\begin{cases}
    C_1(\gamma,d)+2^{1+\gamma+d}C_0,& 0<t\leqslant 1\\
    C_2(\gamma,d,t)+2^{1+\gamma+d}C_0, & t>1
    \end{cases}\,.
\end{align*}
\end{lemma}
By taking $m=\frac{1}{\varepsilon}$ in Lemma~\ref{lem:G1_v2}, we obtain
\begin{align*}
    \int_{\G_1}\|s_t(x)\|^2\rmd x\lesssim C_\gamma(t)^{\varepsilon}\varepsilon^{-1}t^{-1}\rho_n^{-\varepsilon-\frac{d}{1+\gamma+d}}\,.
\end{align*}
Plugging this back into the inequality~\eqref{eq:term1_1} gives
\begin{align*}
    \text{I}&\lesssim (C_\gamma(t)\rho_n)^{-\frac{d}{1+\gamma+d}}\cdot\dfrac{1}{nt(2\pi t)^{d/2}}+(C_\gamma(t)^{\varepsilon}\varepsilon^{-1}t^{-1}\rho_n^{-\varepsilon-\frac{d}{1+\gamma+d}})\cdot\dfrac{1}{n(2\pi t)^{d/2}}\\
    &=C_\gamma(t)^{-\frac{d}{1+\gamma+d}}(t\log n)^{-1}\rho_n^{\frac{1+\gamma}{1+\gamma+d}}+(C_\gamma(t)^{\varepsilon}\varepsilon^{-1}\rho_n^{-\varepsilon})(t\log n)^{-1}\rho_n^{\frac{1+\gamma}{1+\gamma+d}}\\
    &=\left(C_\gamma(t)^{-\frac{d}{1+\gamma+d}}+C_\gamma(t)^{\varepsilon}\varepsilon^{-1}\rho_n^{-\varepsilon}\right)(t\log n)^{-1}\rho_n^{\frac{1+\gamma}{1+\gamma+d}}\,.
\end{align*}

\noindent\textbf{Step 2. } Recall the definition of $\G_2$ and $\hat s_t(x)$, we have
\begin{align*}
    \text{II}\leqslant\int_{\G_2}\E\left[\|s_t(x)\|^2\mathbbm 1\{\hat p_t(x)<\rho_n\}\right]p_t(x)\,\rmd x +\int_{\G_2}\E\left[\|\hat s_t(x)-s_t(x)\|^2\mathbbm 1\{\hat p_t(x)\geqslant \rho_n\}\right]p_t(x)\,\rmd x \,.
\end{align*}
\begin{lemma}\label{lem:G2st_v2}
Under Assumption~\ref{asm:p0tail}, it holds that
\begin{align*}
    \int_{\G_2}p_t(x)\,\rmd x\lesssim C_\gamma(t)^{\frac{d}{d+\gamma+1}}\rho_n^{\frac{\gamma+1}{d+\gamma+1}}\,, 
\end{align*}
and
\begin{align*}
    \int_{\G_2}\|s_t(x)\|^2p_t(x)\,\rmd x\lesssim \varepsilon^{-1}t^{-1}C_\gamma(t)^{\frac{d(1-\varepsilon)}{d+\gamma+1}}\rho_n^{\frac{(\gamma+1)(1-\varepsilon)}{d+\gamma+1}}\,.
\end{align*}
\end{lemma}
Applying Lemma~\ref{lem:G2st_v2}, the first term of $\text{II}$ can be bounded by
\begin{align*}
    \int_{\G_2}\E\left[\|s_t(x)\|^2\mathbbm 1\{\hat p_t(x)<\rho_n\}\right]p_t(x)\,\rmd x\leqslant\int_{\G_2}\|s_t(x)\|^2p_t(x)\,\rmd x\lesssim \varepsilon^{-1}t^{-1}C_\gamma(t)^{\frac{d(1-\varepsilon)}{d+\gamma+1}}\rho_n^{\frac{(\gamma+1)(1-\varepsilon)}{(d+\gamma+1)}}\,.
\end{align*}
For the second term of $\text{II}$, by Proposition~\ref{prop:GaussMSEpt} and Proposition~\ref{prop:GaussMSEhatpt}, and following the similar approach in \eqref{eq:term1}, we have
\begin{align*}
    &\quad\int_{\G_2}\E\left[\|\hat s_t(x)-s_t(x)\|^2\mathbbm 1\{\hat p_t(x)\geqslant \rho_n\}\right]p_t(x)\,\rmd x\\
    &\leqslant \int_{\G_2}2\E\left[\dfrac{\|\nabla\hat p_t(x)-\nabla p_t(x)\|^2+\|s_t(x)\|^2|\hat p_t(x)-p_t(x)|^2}{\hat p_t(x)^2}\mathbbm 1\{\hat p_t(x)\geqslant\rho_n\}\right]p_t(x)\,\rmd x\\
    &\leqslant 2\rho_n^{-2}\int_{\G_2}\left(\text{MSE}(\nabla\hat p_t(x))+\text{MSE}(\hat p_t(x))\cdot\|s_t(x)\|^2\right)p_t(x)\,\rmd x\\
    &\leqslant 2\rho_n^{-2}\left(\int_{\G_2}\dfrac{p_t(x)}{nt(2\pi t)^{d/2}}p_t(x)\,\rmd x+\int_{\G_2}\dfrac{p_t(x)}{n(2\pi t)^{d/2}}\cdot\|s_t(x)\|^2p_t(x)\,\rmd x\right)\\
    &\leqslant 2\rho_n^{-1}\left(\dfrac{c_\alpha}{nt(2\pi t)^{d/2}}\int_{\G_2}p_t(x)\,\rmd x+\dfrac{c_\alpha}{n(2\pi t)^{d/2}}\int_{\G_2}\|s_t(x)\|^2p_t(x)\,\rmd x\right)\,.
\end{align*}
Applying Lemma~\ref{lem:G2st_v2} again, it then follows that
\begin{align*}
    &\quad\int_{\G_2}\E\left[\|\hat s_t(x)-s_t(x)\|^2\mathbbm 1\{\hat p_t(x)\geqslant \rho_n\}\right]p_t(x)\,\rmd x\\
    &\lesssim 2\rho_n^{-1}\left(\dfrac{c_\alpha}{nt(2\pi t)^{d/2}}C_\gamma(t)^{\frac{d}{d+\gamma+1}}\rho_n^{\frac{\gamma+1}{d+\gamma+1}}+\dfrac{c_\alpha}{n(2\pi t)^{d/2}}\varepsilon^{-1}t^{-1}C_\gamma(t)^{\frac{d(1-\varepsilon)}{d+\gamma+1}}\rho_n^{\frac{(\gamma+1)(1-\varepsilon)}{d+\gamma+1}}\right)\\
    &=2c_\alpha(t\log n)^{-1}\left(C_\gamma(t)^{\frac{d}{d+\gamma+1}}\rho_n^{\frac{\gamma+1}{d+\gamma+1}}+\varepsilon^{-1}C_\gamma(t)^{\frac{d(1-\varepsilon)}{d+\gamma+1}}\rho_n^{\frac{(\gamma+1)(1-\varepsilon)}{d+\gamma+1}}\right)
\end{align*}

\noindent\textbf{Step 3. } By \Holder's inequality, we obtain that
\begin{align*}
    \text{III}&\leqslant \left(\int_{\R^d}\E\left[\|\hat s_t(x)-s_t(x)\|^4\right]p_t(x)\,\rmd x\right)^{1/2}\left(\int_{\R^d}\E\left[\mathbbm 1\{x\in A_\alpha^c\}\right]p_t(x)\,\rmd x\right)^{1/2}\\
    &\lesssim n^{-\alpha/2}\left(\int_{\R^d}\|s_t(x)\|^4p_t(x)\,\rmd x+\int_{\R^d}\E\left[\|\hat s_t(x)\|^4\right]p_t(x)\,\rmd x\right)^{1/2}\,.
\end{align*}
\begin{lemma}[Lemma F.16 in~\cite{zhang2024minimax}]\label{lem:Im}
    Let $X$ and $G$ be independent $d$-dimensional random vectors such that $X\sim p_0, G\sim\mathcal N(0,I_d)$ and $Y=X+\sqrt{t}G$. Then, 
    \begin{align*}
        \E[\|\nabla\log p_Y(Y)\|^m]\leqslant t^{-\frac{m}{2}}d^{\frac{m}{2}}(m-1)!! \,.
    \end{align*}
\end{lemma}
By Lemma~\ref{lem:Im}, the first term is bounded by
\begin{align*}
    \int_{\R^d}\|s_t(x)\|^4p_t(x)\,\rmd x\leqslant 3t^{-2}d^2\,,
\end{align*}
and the second term can be bounded by
\begin{align*}
    \int_{\R^d}\E\|\hat s_t(x)\|^4p_t(x)\,\rmd x&=\int_{\R^d}\E\left[\|\hat s_t(x)\|^4\mathbbm 1\{\hat p_t(x)>\rho_n\}\right]p_t(x)\,\rmd x\\
    &\leqslant \rho_n^{-4}\int_{\R^d}\E\left[\|\nabla\hat p_t(x)\|^4\right]p_t(x)\,\rmd x\,.
\end{align*}
Recall the definition of $\nabla p_t(x)$, it holds that
\begin{align*}
    \E\left[\|\nabla\hat p_t(x)\|^4\right]&=\dfrac{1}{(nt)^4}\E_{X^{(i)}\sim p_0}\bigg[\bigg\|\sum_{i=1}^n(X^{(i)}-x)\varphi_t(X^{(i)}-x)\bigg\|^4\bigg]\\
    &\leqslant \frac{1}{t^4}\E_{X\sim p_0}\left[\|(X-x)\varphi_t(X-x)\|^4\right]\\
    &\leqslant \frac{1}{t^4(2\pi)^{d/2}}\sup_{r>0}r^2\exp\left(-2r/t\right)\\
    &\leqslant \frac{1}{t^2e^2(2\pi)^{d/2}}\,.
\end{align*}
Therefore,
\begin{align*}
    \text{III}\lesssim n^{-\alpha/2}\left(3t^{-2}d^2+\rho_n^{-4}[t^2e^2(2\pi)^{d/2}]^{-1}\right)^{1/2}\,.
\end{align*}
We can choose $\alpha$ large enough to cover the order of $t$ and $n$.

Combining Step 1-3, we obtain
\begin{align*}
    \E\left[\int_{\R^d}\|\hat s_t(x)-s_t(x)\|^2p_t(x)\,\rmd x\right]&\lesssim \left(C_\gamma(t)^{-\frac{d}{1+\gamma+d}}+C_\gamma(t)^{\varepsilon}\varepsilon^{-1}\rho_n^{-\varepsilon}\right)(t\log n)^{-1}\rho_n^{\frac{1+\gamma}{1+\gamma+d}}\\
    &+2c_\alpha(t\log n)^{-1}\left(C_\gamma(t)^{\frac{d}{d+\gamma+1}}\rho_n^{\frac{\gamma+1}{d+\gamma+1}}+\varepsilon^{-1}C_\gamma(t)^{\frac{d(1-\varepsilon)}{d+\gamma+1}}\rho_n^{\frac{(\gamma+1)(1-\varepsilon)}{d+\gamma+1}}\right)\\
    &+\varepsilon^{-1}t^{-1}C_\gamma(t)^{\frac{d(1-\varepsilon)}{d+\gamma+1}}\rho_n^{\frac{(\gamma+1)(1-\varepsilon)}{(d+\gamma+1)}}\,
\end{align*}
as desired.
\end{proof}

\subsubsection{Proof of Theorem~\ref{thm:early_stopping}}

\begin{proof}
Firstly, the total variation distance between $p_0$ and $p_{t_0}$ can be decomposed into a central part and a tail part by introducing a truncation radius $R > 0$. By Jensen's inequality, it holds that
\begin{equation}
\begin{aligned}\label{eq:early_stopping}
\operatorname{TV}(p_0, p_{t_0}) &= \frac{1}{2} \int_{\mathbb{R}^d} |p_0(x) - p_{t_0}(x)| \,\rmd x \\
&= \frac{1}{2} \int_{\|x\| \leqslant R} |p_0(x) - p_{t_0}(x)| \,\rmd x + \frac{1}{2} \int_{\|x\| > R} |p_0(x) - p_{t_0}(x)| \,\rmd x \\
&\leqslant \frac{1}{2} \sqrt{\text{Vol}(\mathbb B_R)} \sqrt{\int_{\mathbb{R}^d} |p_0(x) - p_{t_0}(x)|^2 \,\rmd x} + \frac{1}{2} \int_{\|x\| > R} (p_0(x) + p_{t_0}(x)) \,\rmd x \\
&\lesssim R^{d/2} \|p_0 - p_{t_0}\|_{L^2} + \int_{\|x\| > R} p_0(x) \,\rmd x + \int_{\|x\| > R} p_{t_0}(x) \,\rmd x\,.
\end{aligned}
\end{equation}
For the second term, we have
\begin{align*}
    \int_{\|x\| > R} p_0(x) \,\rmd x \leqslant C_0 \int_R^\infty r^{-(d+\gamma+1)} r^{d-1} \,\rmd r = C_0\int_R^\infty r^{-\gamma-2} \,\rmd r \lesssim C_0R^{-(\gamma+1)}\,.
\end{align*}
For the second tail term, recall that $X_{t_0} = X_0 + \sqrt{t_0}Z$, where $X_0 \sim p_0$ and $Z \sim \mathcal{N}(0, I_d)$. For sufficiently large $R$ (relative to $\sqrt{t_0}$), the heavy tail of $X_0$ dominates the Gaussian noise. Specifically,
\begin{align*}
    \mathbb{P}(\|X_{t_0}\| > R) &\leqslant \mathbb{P}(\|X_0\| > R/2) + \mathbb{P}(\|\sqrt{t_0}Z\| > R/2) \\
    &\leqslant C_0(R/2)^{-(\gamma+1)} + \exp\left(-\frac{R^2}{8t_0}\right)\,.
\end{align*}
Since polynomial decay dominates exponential decay for large $R$, we have:
\begin{align}
    \text{Tail Error} \lesssim R^{-(\gamma+1)}\,. \label{eq:tail_bound}
\end{align}
The following step follows Plancherel's theorem,
\begin{align*}
    \|p_0 - p_{t_0}\|_{L^2}^2 &= \frac{1}{(2\pi)^d} \int_{\mathbb{R}^d} |\mathcal{F}[p_0](\omega) - \mathcal{F}[p_{t_0}](\omega)|^2 \,\rmd \omega \\
    &= \frac{1}{(2\pi)^d} \int_{\mathbb{R}^d} |\mathcal{F}[p_0](\omega)|^2 |1 - \phi_{t_0}(\omega)|^2 \,\rmd \omega\,,
\end{align*}
where $\phi_{t_0}(\omega) = \exp(-\frac{t_0 \|\omega\|^2}{2})$. We split the integral into low and high frequencies at threshold $K = t_0^{-1/2}$.

Note that $|1 - \phi_{t_0}(\omega)| \leqslant 2$, it then follows that
\begin{align*}
    \int_{\|\omega\| \geqslant t_0^{-1/2}} |\mathcal{F}[p_0]|^2 |1 - \phi_{t_0}|^2\,\rmd\omega &\leqslant 4 \int_{\|\omega\| \geqslant t_0^{-1/2}} |\mathcal{F}[p_0]|^2 \frac{\|\omega\|^{2\beta}}{\|\omega\|^{2\beta}} \,\rmd\omega\\
    &\leqslant 4 (t_0^{-1/2})^{-2\beta} \int_{\mathbb{R}^d} \|\omega\|^{2\beta} |\mathcal{F}[p_0]|^2 \,\rmd\omega\\
    &\lesssim t_0^\beta L^2\,.
\end{align*}
Since $|1 - e^{-x}|^2 \leqslant x^2$ for $x \geqslant 0$, we then have $|1 - \phi_{t_0}(\omega)|^2 \leqslant (\frac{t_0 \|\omega\|^2}{2})^2 = \frac{t_0^2 \|\omega\|^4}{4}$. Thus,
\begin{align*}
    \int_{\|\omega\| < t_0^{-1/2}} |\mathcal{F}[p_0]|^2 |1 - \phi_{t_0}|^2\,\rmd\omega &\lesssim t_0^2 \int_{\|\omega\| < t_0^{-1/2}} |\mathcal{F}[p_0]|^2 \|\omega\|^4 \,\rmd\omega\\
    &= t_0^2 \int_{\|\omega\| < t_0^{-1/2}} |\mathcal{F}[p_0]|^2 \|\omega\|^{2\beta} \|\omega\|^{4-2\beta}\,\rmd\omega\,.
\end{align*}
When $\beta \leqslant 2$, it holds that $4-2\beta \geqslant 0$. Thus, $\|\omega\|^{4-2\beta} \leqslant (t_0^{-1/2})^{4-2\beta} = t_0^{\beta-2}$, and
\begin{align*}
    \int_{\|\omega\| < t_0^{-1/2}} |\mathcal{F}[p_0]|^2 |1 - \phi_{t_0}|^2\,\rmd\omega \lesssim t_0^2 \cdot t_0^{\beta-2} \int_{\mathbb{R}^d} \|\omega\|^{2\beta} |\mathcal{F}[p_0]|^2 \lesssim t_0^\beta L^2\,.
\end{align*}
Combining high and low frequencies, we obtain
\begin{align}
    \|p_0 - p_{t_0}\|_{L^2} \lesssim \sqrt{t_0^\beta} = t_0^{\beta/2}\,. \label{eq:l2_bound}
\end{align}

\noindent\textbf{Step 3: Optimal Radius Selection.}
Plugging \eqref{eq:tail_bound} and \eqref{eq:l2_bound} back into the display \eqref{eq:early_stopping} yields
\begin{align*}
    \operatorname{TV}(p_0, p_{t_0}) \lesssim R^{d/2} t_0^{\beta/2} + R^{-(\gamma+1)}\,.
\end{align*}
To optimize this bound, we balance the two contributions by choosing $R$ such that
\begin{align*}
    R^{\frac{d}{2} + \gamma + 1} \asymp t_0^{-\beta/2}\,, %
\end{align*}
which yields the choice
\begin{align*}
R \asymp t_0^{-\frac{\beta}{d + 2(\gamma+1)}}\,.
\end{align*}
Plugging this optimal $R$ back into the previous display gives
\begin{align*}
    \operatorname{TV}(p_0, p_{t_0}) \lesssim \left( t_0^{-\frac{\beta}{d + 2(\gamma+1)}} \right)^{-(\gamma+1)} = t_0^{\frac{\beta(\gamma+1)}{d + 2(\gamma+1)}}
\end{align*}
as desired.
\end{proof}

\subsubsection{Proof of Theorem~\ref{thm:poly-sampling}}
{
\begin{proof}
    Recall that the backward process $X^\leftarrow$ is approximated by $\hat Y_t$. Let $(\tilde Y_t)_{t\in[0,T]}$ be $(\hat Y_t)_{t\in[0,T]}$ replacing $\hat Y_0\sim \mathcal N(0,TI_d)$ by $\tilde Y_0\sim p_T$, that is, $(\tilde Y_t)_{t\in[0,T]}$ satisfies:
    \begin{align*}
        \rmd \tilde Y_t=\hat s_{T-t}(\tilde Y_t)\,\rmd t+\rmd W_t,\quad \tilde Y_t\sim p_T\,.
    \end{align*}
    By the triangle inequality, it holds that
    \begin{align*}
        &\quad\E\left[\operatorname{TV}(\law(X_0),\law(\hat Y_{T-t_0}))\right]\\
        &\leqslant \operatorname{TV}(\law(X_0),\law(X_{T-t_0}^\leftarrow))+\E[\operatorname{TV}(\law(X_{T-t_0}^\leftarrow),\law(\tilde{Y}_{T-t_0}))]+\E[\operatorname{TV}(\law(\tilde Y_{T-t_0}),\law(\hat Y_{T-t_0}))]\,.
    \end{align*}
    By Theorem~\ref{thm:early_stopping}, $\operatorname{TV}(\law(X_0),\law(X_{T-t_0}^\leftarrow))\lesssim n^{-\frac{2\beta(\gamma+1)}{4\beta(d+\gamma+1)+d(d+2(\gamma+1))}}$. For the second term, by Pinsker's inequality and data-processing inequality, we obtain
    \begin{align*}
        \operatorname{TV}(\law(X_{T-t_0}^\leftarrow),\law(\tilde Y_{T-t_0}))\lesssim\sqrt{\operatorname{KL}(\law(X_{T-t_0}^\leftarrow)\|\law(\tilde Y_{T-t_0}))}\leqslant\sqrt{\operatorname{KL}(\P_{X^\leftarrow}\|\P_{\tilde Y})}
    \end{align*}
    where $\P_{X^\leftarrow}$ and $\P_{\tilde Y}$ are the path measure for $(X_t^\leftarrow)_{t\in[0,T-t_0]}$ and $(\tilde Y_t)_{t\in[0,T-t_0]}$. Then, we have
    \begin{align*}
        \operatorname{KL}(\P_{X^\leftarrow}\|\P_{\tilde Y})&\leqslant \dfrac{1}{2}\int_0^{T-t_0}\int_{\R^d}\|\hat s_{T-t}(x)-s_{T-t}(x)\|^2\,\rmd x\,\rmd t\\
        &=\dfrac{1}{2}\int_{t_0}^T\int_{\R^d}\|\hat s_t(x)-s_t(x)\|^2\,\rmd x\,\rmd t
    \end{align*}
    Then $\E[\text{KL}(\P_{X^\leftarrow}\|\P_{\tilde Y})]$ can be bounded by Theorem~\ref{thm:t0scoreerror}, and by Jenson's inequality, we have
    \begin{align*}
        \E[\operatorname{TV}(\law(X_{T-t_0}^\leftarrow),\law(\tilde Y_{T-t_0}))]\lesssim\E\left[\sqrt{\operatorname{KL}(\P_{X^\leftarrow}\|\P_{\tilde Y})}\right]\leqslant \sqrt{\E[\operatorname{KL}(\P_{X^\leftarrow}\|\P_{\tilde Y})]}\,.
    \end{align*}
    It follows that 
    \begin{align*}
        \E[\operatorname{TV}(\law(X_{T-t_0}^\leftarrow),\law(\tilde Y_{T-t_0}))]\lesssim \text{polylog}(n)n^{-\frac{\gamma+1}{2(d+\gamma+1)}}t_0^{-\frac{d(\gamma+1)}{4(d+\gamma+1)}}\asymp \text{polylog}(n)n^{-\frac{2\beta(\gamma+1)}{4\beta(d+\gamma+1)+d(d+2(\gamma+1))}}\,.
    \end{align*}
    Since the only difference of $\tilde Y_{T-t_0}$ and $\hat Y_{T-t_0}$ is their initial distribution, we have
    \begin{align*}
        \E[\operatorname{TV}(\law(\tilde Y_{T-t_0}),\law(\hat Y_{T-t_0}))]\leqslant \operatorname{TV}(\law(X_T),\mathcal N(0,TI_d))\lesssim \dfrac{1}{\sqrt{T}}\,.
    \end{align*}
    The last inequality follows from Lemma~\ref{lem:TV}.
    By taking $T=n^{\frac{4\beta(\gamma+1)}{4\beta(d+\gamma+1)+d(d+2(\gamma+1))}}$, we reach at
    \begin{align*}
        \E[\operatorname{TV}(\law(\tilde Y_{T-t_0}),\law(\hat Y_{T-t_0}))]\lesssim n^{-\frac{2\beta(\gamma+1)}{4\beta(d+\gamma+1)+d(d+2(\gamma+1))}}\,.
    \end{align*}
    This completes the proof.
\end{proof}
}

{

\begin{lemma}\label{lem:TV} Let $p_0$ be a probability density on $\R^d$ with finite first moment
$M_1:=\int_{\R^d}\|z\|\,p_0(z)\,\rmd z<\infty.$
Let $\varphi_T$ denote the density of $\mathcal N(0,TI_d)$, and define $p_T:=p_0*\varphi_T$.
Then, it holds that
\begin{align*}
\operatorname{TV}(p_T,\varphi_T)\leqslant \frac{M_1}{2\sqrt{T}}.
\end{align*}
\end{lemma}
}

{

\subsection{Proof of Section~\ref{sec:exp_decay}}
Proposition~\ref{prop:GaussMSEpt},~\ref{prop:GaussMSEhatpt} and Lemma~\ref{lem:Gausspthat} are independent of the behavior of tail decay, so they remain valid without change.
For the global convergence theorem for score estimation, we revisit the proof step by step: only Lemmas~\ref{lem:G1_v2},~\ref{lem:G2st_v2}, and~\ref{lem:pttail} require modification, after which the desired result follows immediately.
\begin{lemma}[Modification of Lemma~\ref{lem:pttail}]\label{lem:pt_expdecay}
Under Assumption~\ref{asm:heavytail}, it holds that
\begin{align*}
    p_t(x)\leqslant \mathscr C_\gamma(t)\exp\left(-c_\gamma(t)\min (\|x\|^2,\|x\|^\gamma)\right) \,.
\end{align*}
where $\mathscr C_\gamma(t):=\frac{1}{(2\pi t)^{d/2}}+C_1$ and $c_\gamma(t):=\min\left(\frac{1}{8t},\frac{c_1}{2^\gamma}\right)$.
\end{lemma}
\subsubsection{Proof of Theorem~\ref{thm:mse_exp}}
\begin{proof}
The proof follows the same threshold-based decomposition strategy as the proof of Theorem~\ref{thm:MSEst}. 
We decompose the weighted MISE into contributions from the high-density bulk region $\mathcal{G}_1$ and the low-density tail region $\mathcal{G}_2$.

Recall the decomposition from \eqref{eq:term1} and \eqref{eq:term1_1}.
On the bulk region $\mathcal{G}_1$, the error is dominated by the variance term, which scales as $|\mathcal{G}_1| \cdot \frac{1}{n(2\pi t)^{d/2}}$. Under Assumption~\ref{asm:heavytail} (exponential decay), 
the following lemma adapting Lemma~\ref{lem:G1_v2} shows that the size of $\G_1$ grows only logarithmically.
\begin{lemma}[Modification of Lemma~\ref{lem:G1_v2}]\label{lem:G1_expdecay}
    Under Assumption \ref{asm:heavytail}, for the threshold $\rho_n=\frac{\log n}{n(2\pi t)^{d/2}}$, the volume of the high-density region $\G_1 = \{x : p_t(x) \geqslant c_\alpha\rho_n\}$ satisfies
    \begin{align*}
        |\G_1| \lesssim \left( \log(1/\rho_n) \right)^{d/\gamma}
    \end{align*}
    and 
    \begin{align*}
        \int_{\G_1}\|s_t(x)\|^2\rmd x\lesssim {\rm{polylog}}(nt)\cdot \varepsilon t^{-1}\rho_n^{-\varepsilon}\,.
    \end{align*}
\end{lemma}
Consequently, the integrated error over the bulk satisfies
\[
\text{I} \lesssim   \operatorname{polylog}(n)\frac{1}{n t^{1+d/2}}.
\]
For the tail region $\mathcal{G}_2$, the error is controlled by the probability mass of the tail.
\begin{lemma}[Modification of Lemma~\ref{lem:G2st_v2}]\label{lem:tail_prob_exp}
    Under Assumption \ref{asm:heavytail}, for the threshold $\rho_n=\frac{\log n}{n(2\pi t)^{d/2}}$, the probability mass of the tail region $\G_2 = \{x : p_t(x) < c_\alpha\rho_n\}$ satisfies
    \begin{align}
        \int_{\G_2}p_t(x)\,\rmd x\lesssim \rho_n(\log(1/\rho_n))^{d/k}
    \end{align}
    and 
    \begin{align*}
        \int_{\G_2}\|s_t(x)\|^2p_t(x)\,\rmd x\lesssim \varepsilon^{-1}t^{-1}\rho_n^{1-\varepsilon}\,.
    \end{align*}
\end{lemma}
Due to the exponential decay, this mass decays super-polynomially fast relative to the threshold $\rho_n$, rendering the bias and truncation errors negligible compared to the variance term in the bulk.

Combining these parts, the total MSE is dominated by the parametric variance rate scaled by the effective volume of the bulk
\[
\mathbb{E}\left[\int_{\mathbb{R}^d}\|\hat{s}_t(x)-s_t(x)\|^2 p_t(x) \rmd x\right] \lesssim \operatorname{polylog}(n) n^{-1} t^{-(1+d/2)}.
\]
\end{proof}

\subsubsection{Proof of Theorem~\ref{thm:exp_score}}
\begin{proof}
We integrate the pointwise MSE bound established in Theorem~\ref{thm:mse_exp} over the diffusion time interval $[t_0, T]$
\begin{align*}
\int_{t_0}^T \mathbb{E}\left[\|\hat{s}_t - s_t\|_{L^2(p_t)}^2\right] \rmd t 
&\lesssim \int_{t_0}^T \operatorname{polylog}(n) n^{-1} t^{-(1+d/2)} \rmd t \\
&= \operatorname{polylog}(n) n^{-1} \left[ \frac{t^{-d/2}}{-d/2} \right]\Bigg|_{t_0}^T \\
&\lesssim \operatorname{polylog}(n) n^{-1} t_0^{-d/2}.
\end{align*}
This yields the desired result.
\end{proof}

\subsubsection{Proof of Theorem~\ref{thm:early_stopping_exp}}
\begin{proof}
The proof proceeds analogously to that of Theorem~\ref{thm:early_stopping}, utilizing a truncation radius $R$ to decompose the total variation distance.
\begin{align*}
    \operatorname{TV}(p_0, p_{t_0}) \leqslant \frac{1}{2} R^{d/2} \|p_0 - p_{t_0}\|_{L^2} + \frac{1}{2} \int_{\|x\| > R} (p_0(x) + p_{t_0}(x)) \rmd x.
\end{align*}
For the smoothing error in the central region, the Sobolev assumption yields the same bound as in the polynomial case (derived via Plancherel's theorem)
\[
\|p_0 - p_{t_0}\|_{L^2} \lesssim t_0^{\beta/2}.
\]
For the tail region, under Assumption 1, the mass decays exponentially. Specifically, for sufficiently large $R$,
\[
\int_{\|x\| > R} p_0(x) \rmd x \leqslant C \exp(-c R^\gamma),
\]
and the convolution $p_{t_0}$ satisfies a similar tail bound.
We choose the radius $R$ such that the tail error is comparable to the smoothing error (up to logarithmic factors), i.e., $\exp(-c R^\gamma) \approx t_0^{\beta/2}$. This implies:
\[
R \asymp \left(\log(1/t_0)\right)^{1/\gamma}.
\]
Plugging this radius back into the decomposition, the volume factor contributes a polylogarithmic term $R^{d/2} \asymp (\log(1/t_0))^{d/2\gamma}$. The total error is thus dominated by the smoothing term
\[
\operatorname{TV}(p_0, p_{t_0}) \lesssim t_0^{\beta/2} (\log(1/t_0))^{\frac{d}{2\gamma}}.
\]
\end{proof}

\subsubsection{Proof of Theorem~\ref{thm:exp_sampling}}
\begin{proof}
The proof proceeds by decomposing the total variation error into three components—early stopping bias, score matching error, and initialization error—following the identical argument as in Theorem 4.
\begin{align*}
    \mathbb{E}[\operatorname{TV}(p_0, \mathcal{L}(\hat{Y}_{T-t_0}))] \le \underbrace{\operatorname{TV}(p_0, p_{t_0})}_{\text{I}} + \underbrace{\mathbb{E}[\operatorname{TV}(p_{t_0}, \mathcal{L}(\tilde{Y}_{T-t_0}))]}_{\text{II}} + \underbrace{\operatorname{TV}(p_T, \mathcal{N}(0, TI_d))}_{\text{III}}.
\end{align*}
For the early stopping bias I, Theorem~\ref{thm:early_stopping_exp} establishes that under exponential tails, the error is dominated by the Sobolev smoothing term $t_0^{\beta/2}$.
For the score matching error II, applying Pinsker's inequality and the cumulative MSE bound from Theorem~\ref{thm:exp_score} yields a rate of $\operatorname{polylog}(n) n^{-1/2} t_0^{-d/4}$.
For the initialization error III, Lemma~\ref{lem:TV} ensures a decay of $T^{-1/2}$, which vanishes for large $T$.
Balancing the bias $t_0^{\beta/2}$ against the score error $n^{-1/2} t_0^{-d/4}$ yields the optimal stopping time $t_0 \asymp n^{-\frac{2}{2\beta+d}}$ and the stated minimax convergence rate.
\end{proof}
}

\subsection{Proof of Auxiliary Lemma for Upper Bounds}
\subsubsection{Proof of Lemmas for Polynomial Decay}
\begin{proof}[Proof of Lemma~\ref{lem:Gausspthat}]
    Note that $\hat p_t$ is an unbiased estimator, and Proposition~\ref{prop:GaussMSEpt} derives the upper bound of $\text{MSE}(\hat p_t)$,
    \begin{align*}
        \text{MSE}(\hat p_t(x))\leqslant \dfrac{p_t(x)}{n(2\pi t)^{d/2}}\,.
    \end{align*}
   It also holds that
    \begin{align*}
        \max_{1\leqslant i\leqslant n}\dfrac{1}{n}\left(\varphi_t(X^{(i)}-x)-\E[\varphi_t(X^{(i)}-x)]\right)\leqslant \dfrac{2}{n}\|\varphi_t(x)\|_\infty=\dfrac{2}{n(2\pi t)^{d/2}}\,.
    \end{align*}
    By Bernstein's Inequality, we have
    \begin{align*}
        \P\left(|\hat p_t(x)-p_t(x)|\geqslant \varepsilon\right)\leqslant 2\exp\left(-\dfrac{\varepsilon^2}{2p_t(x)/(n(2\pi t)^{d/2})+4\varepsilon/(3n(2\pi t)^{d/2})}\right)\,.
    \end{align*}
    Then, with probability at least $1-\delta$, it holds that
    \begin{align*}
        \left|\hat p_t(x)-p_t(x)\right|<\max\left\{\sqrt{\dfrac{4p_t(x)\log(2/\delta)}{n(2\pi t)^{d/2}}},\dfrac{8\log(2/\delta)}{3n(2\pi t)^{d/2}}\right\}\,.
    \end{align*}
    Let $\delta=n^{-\alpha}$, we obtain that
    \begin{align*}
        \P\bigg(|\hat p_t(x)-p_t(x)|<\max\left\{\sqrt{\dfrac{4p_t(x)\log(2n)\alpha}{n(2\pi t)^{d/2}}},\dfrac{8\log(2n)\alpha}{3n(2\pi t)^{d/2}}\right\}\bigg)\geqslant 1-n^{-\alpha}
    \end{align*}
   as desired. 
\end{proof}

To proceed to proof of Lemma~\ref{lem:G1_v2} and~\ref{lem:G2st_v2}, we need the following lemma
\begin{lemma}\label{lem:pttail}
Under Assumption~\ref{asm:p0tail}, when $\|x\|>1$, it holds that
\begin{align*}
    p_t(x)\leqslant C_\gamma(t)(1+\|x\|^2)^{-(1+d+\gamma)/2}
\end{align*}
where
\begin{align*}
    C_\gamma(t)=\begin{cases}
    C_1(\gamma,d)+2^{1+\gamma+d}C_0,& 0<t\leqslant 1\\
    C_2(\gamma,d,t)+2^{1+\gamma+d}C_0, & t>1
    \end{cases}\,.
\end{align*}
\end{lemma}
\begin{proof}[Proof of Lemma~\ref{lem:pttail}]
    Let $R=\|x\|/2$. We decompose the density $p_t$
\begin{align*}
    p_t(x)=p_0*\varphi_t(x)=\underbrace{\int_{\|y\|\leqslant R}p_0(y)\varphi_t(x-y)\,\rmd y}_{\text{I}}+\underbrace{\int_{\|y\|>R}p_0(y)\varphi_t(x-y)\,\rmd y}_{\text{II}}
\end{align*}
The first term on the right side can be bounded as follows
\begin{align*}
    \text I&\leqslant \int_{\|y\|\leqslant R}p_0(y)\sup_{\|y\|\leqslant R}\varphi_t(x-y)\,\rmd y\\
    &=\int_{\|y\|\leqslant R}p_0(y)\dfrac{1}{(2\pi t)^{d/2}}\exp\left(-\dfrac{\|x\|^2}{8t}\right)\rmd y\\
    &\leqslant \dfrac{1}{(2\pi t)^{d/2}}\exp\left(-\dfrac{\|x\|^2}{8t}\right)\,.
\end{align*}
Denote 
\begin{align*}
    F(t,r):=\dfrac{1}{(2\pi t)^{d/2}}\exp(-r/8t)(1+r)^{(1+\gamma+d)/2}\,.
\end{align*}
For fixed $r\geqslant 1$, the maximum point of $F(t,r)$ is $t(r)=r/(4d)$, it then follows that 
\begin{align*}
    \sup_{0<t\leqslant 1,r\geqslant 1}F(t,r)&=\max\left\{\sup_{r\leqslant 4d}\left(\frac{2d}{\pi}\right)^{d/2}e^{-d/2}(1+r)^{\frac{1+\gamma+d}{2}}r^{-d/2},\sup_{r>4d}(2\pi)^{-d/2}e^{-r/8}(1+r)^{\frac{1+\gamma+d}{2}}\right\}\\
    &\leqslant\max\left\{\left(\dfrac{2d^2}{\pi e(1+\gamma)}\right)^{d/2}\left(\dfrac{1+\gamma+d}{d}\right)^{\frac{1+\gamma+d}{2}},(2\pi)^{-d/2}e^{-\frac{3+4\gamma+4d}{8}}[4(1+\gamma+d)]^{\frac{1+\gamma+d}{2}}\right\}\\
    &:=C_1(\gamma,d)
\end{align*}
where $C_1(\gamma,d)$ is independent of $t$. For fixed $t>1$, the maximum point of $F(t,r)$ is $r(t)=4t(1+\gamma+d)-1$, then
\begin{align*}
\sup_{r\geqslant 1}F(t,r)&=\frac{t^{(1+\gamma)/2}}{(2\pi)^{d/2}}\exp\left(-\frac{4t(1+\gamma+d)-1}{8t}\right)[4(1+\gamma+d)]^{(1+\gamma+d)/2}\\
&\leqslant \frac{t^{(1+\gamma)/2}}{(2\pi)^{d/2}}e^{-\frac{1+\gamma+d}{4}}[4(1+\gamma+d)]^{\frac{1+\gamma+d}{2}}\\
&:=C_2(\gamma,d,t)
\end{align*}
Therefore, $\text{I}\leqslant C_1(\gamma,d)(1+\|x\|^2)^{-(1+\gamma+d)/2}$ when $0<t\leqslant 1$ and $\text I\leqslant C_2(\gamma,d,t)(1+\|x\|^2)^{-(1+\gamma+d)/2}$ when $t>1$.
The second term
\begin{align*}
    |\text{II}|\leqslant \int_{\|y\|>R} \dfrac{C_0}{(1+\|y\|^2)^{(1+\gamma+d)/2}}\cdot\varphi_t(x-y)\,\rmd y\leqslant \dfrac{2^{1+\gamma+d}C_0}{(1+\|x\|^2)^{(1+\gamma+d)/2}}\,.
\end{align*}
Combining these two terms, we obtain that
\begin{align*}
    p_t(x)\leqslant \dfrac{C_\gamma(t)}{(1+\|x\|^2)^{(1+\gamma+d)/2}}\,,
\end{align*}
where 
\begin{align*}
    C_\gamma(t)=\begin{cases}
    C_1(\gamma,d)+2^{1+\gamma+d}C_0,& 0<t\leqslant 1\\
    C_2(\gamma,d,t)+2^{1+\gamma+d}C_0, & t>1
    \end{cases}\,.
\end{align*}
\end{proof}

\begin{proof}[Proof of Lemma~\ref{lem:G1_v2}]
    By Lemma~\ref{lem:pttail}, we have
    \begin{align*}
        p_t(x)\leqslant \dfrac{C_\gamma(t)}{(1+\|x\|^2)^{(1+\gamma+d)/2}}\,.
    \end{align*}
    Since $\rho_n=\frac{\log n}{n(2\pi t)^{d/2}}$, then $p_t(x)>c_\alpha\rho_n$ implies that
    \begin{align*}
        \|x\|\leqslant \left(C_\gamma(t)/(c_\alpha\rho_n)\right)^{\frac{1}{1+\gamma+d}}\lesssim (C_\gamma(t)\rho_n)^{-\frac{1}{1+\gamma+d}}\,.
    \end{align*}
    Therefore, we obtain
    \begin{align*}
        |\G_1|\leqslant\left|\left\{x:\|x\|\leqslant(C_\gamma(t)\rho_n)^{-\frac{1}{1+\gamma+d}}\right\}\right|\lesssim(C_\gamma(t)\rho_n)^{-\frac{d}{1+\gamma+d}}\,.
    \end{align*}
    By \Holder's inequality, it holds that
    \begin{align*}
        \int_{\G_1}\|s_t(x)\|^2\rmd x&=\int_{G_1}\|s_t(x)\|^2\dfrac{1}{p_t(x)}p_t(x)\,\rmd x\\
        &\leqslant \left(\int_{\G_1}\|s_t(x)\|^{2m}p_t(x)\,\rmd x\right)^{1/m}\left(\int_{\G_1}p_t(x)^{-m'}p_t(x)\,\rmd x\right)^{1/m'}\,,
    \end{align*}
    where $m,m'>1$ and satisfies $\frac{1}{m}+\frac{1}{m'}=1$. By Lemma~\ref{lem:Im}, it holds that
    \begin{align*}
        \left(\int_{\G_1}\|s_t(x)\|^{2m}p_t(x)\,\rmd x\right)^{1/m}\lesssim mt^{-1}\,,
    \end{align*}
    Invoking the bound of $|\G_1|$, we have
    \begin{align*}
        \left(\int_{\G_1}p_t(x)^{-m'}p_t(x)\rmd x\right)^{1/m'}&=\left(\int_{\G_1}p_t(x)^{1-m'}\rmd x\right)^{1/m'}\\
        &\leqslant [|\G_1|(c_\alpha\rho_n)^{1-m'}]^{1/m'}\\
        &\leqslant {C_\gamma(t)}^{-\frac{1}{m}}\rho_n^{-\frac{1}{m}-\frac{d}{(1+\gamma+d)m'}}\,.
    \end{align*}
    Combining these two terms, we have
    \begin{align*}
        \int_{\G_1}\|s_t(x)\|^2\rmd x\lesssim C_\gamma(t)^{-\frac{1}{m}}mt^{-1}\rho_n^{-\frac{1}{m}-\frac{d}{(1+\gamma+d)m'}}\,.
    \end{align*}
\end{proof}

\begin{proof}[Proof of Lemma~\ref{lem:G2st_v2}]
Note that 
\begin{align*}
    \int_{\G_2}p_t(x)\,\rmd x=\P_{X\sim p_t}(p_t(X)\leqslant c_\alpha\rho_n)\,.
\end{align*}
By \Holder's inequality and Lemma~\ref{lem:Im}, it holds that
    \begin{align*}
        \int_{\G_2}\|s_t(x)\|^2p_t(x)\,\rmd x&=\E_{X\sim p_t}\left[\|s_t(X)\|^2\mathbbm 1\{p_t(x)\leqslant c_\alpha\rho_n\}\right]\\
        &\leqslant \E_{X\sim p_t}[\|s_t(X)\|^{2m}]^{\frac{1}{m}}\E\left[\mathbbm 1\{p_t(X)\leqslant c_\alpha\rho_n\}\right]^{1-\frac{1}{m}}\\
        &\lesssim mt^{-1}\P(p_t(X)\leqslant c_\alpha\rho_n)^{1-\frac{1}{m}}\,.
    \end{align*}
Let $R_n=\left(\frac{C_\gamma(t)}{c_\alpha\rho_n}\right)^{\frac{1}{d+\gamma+1}}$, by Lemma~\ref{lem:pttail}, it holds that
\begin{align*}
    \{x:\|x\|\geqslant R_n\}\subset\{x:p_t(x)\leqslant c_\alpha\rho_n\}\,.
\end{align*}
Then, we can decompose the probability
\begin{align}\label{eq:G2Prob}
\begin{aligned}
    \P(p_t(X)\leqslant c_\alpha\rho_n)&=\P(p_t(X)\leqslant c_\alpha\rho_n,\|X\|<R_n)+\P(p_t(X)\leqslant c_\alpha\rho_n,\|X\|\geqslant R_n)\\
    &\leqslant \P(p_t(X)\leqslant c_\alpha\rho_n,\|X\|<R_n)+\P(\|X\|\geqslant R_n)\\
    &=\int_{\|x\|< R_n}p_t(x)\mathbbm 1\{p_t(x)\leqslant c_\alpha\rho_n\}\,\rmd x+\int_{\|x\|\geqslant R_n} \dfrac{C_\gamma(t)}{(1+\|x\|^2)^{(d+\gamma+1)/2}}\rmd x\\
    &\leqslant c_\alpha\rho_n\text{Vol}(\mathbb B_{R_n})+C_\gamma(t)\cdot\dfrac{2\pi^{d/2}}{\Gamma(d/2)}\int_{R_n}^\infty \rho^{-\gamma-2}\rmd \rho\\
    &\leqslant c_\alpha\rho_n\cdot\dfrac{\pi^{d/2}}{\Gamma(d/2+1)}R_n^d+C_\gamma(t)\cdot\dfrac{2\pi^{d/2}}{\Gamma(d/2)}R_n^{-\gamma-1}\\
    &=c_\alpha^{\frac{\gamma+1}{d+\gamma+1}} C_\gamma(t)^{\frac{d}{d+\gamma+1}}\left(\dfrac{\pi^{d/2}}{\Gamma(d/2+1)}+\dfrac{2\pi^{d/2}}{\Gamma(d/2)}\right)\cdot\rho_n^{\frac{\gamma+1}{d+\gamma+1}}\,.
\end{aligned}
\end{align}
Therefore, we obtain
\begin{align*}
    \int_{\G_2}p_t(x)\,\rmd x\lesssim C_\gamma(t)^{\frac{d}{d+\gamma+1}}\rho_n^{\frac{\gamma+1}{d+\gamma+1}}\,,
\end{align*}
and
\begin{align*}
    \int_{\G_2}\|s_t(x)\|^2p_t(x)\,\rmd x\lesssim\varepsilon^{-1}t^{-1}C_\gamma(t)^{\frac{d(1-\varepsilon)}{d+\gamma+1}}\rho_n^{\frac{(\gamma+1)(1-\varepsilon)}{(d+\gamma+1)}}\,
\end{align*}
as desired.
\end{proof}

\begin{proof}[Proof of Lemma~\ref{lem:TV}]
We start with the definition of TV distance
\begin{align*}
\operatorname{TV}(p_T, \varphi_T) = \frac{1}{2} \int_{\mathbb{R}^d} |p_T(x) - \varphi_T(x)| \rmd x.
\end{align*}
Substituting the convolution representation of $p_T$ into the total variation distance and applying the triangle inequality yields
\begin{align*}
|p_T(x) - \varphi_T(x)| &= \left| \int_{\mathbb{R}^d} p_0(z) \varphi_T(x - z) \rmd z - \varphi_T(x) \right| \\
&= \left| \int_{\mathbb{R}^d} p_0(z) \varphi_T(x - z) \rmd z - \int_{\mathbb{R}^d} p_0(z) \varphi_T(x) \rmd z \right| \\
&= \left| \int_{\mathbb{R}^d} p_0(z) \left( \varphi_T(x - z) - \varphi_T(x) \right) \rmd z \right| \\
&\leqslant \int_{\mathbb{R}^d} p_0(z) \left| \varphi_T(x - z) - \varphi_T(x) \right| \rmd z.
\end{align*}
Integrating both sides over $\mathbb{R}^d$ with respect to $x$ and then interchanging the order of integration gives
\begin{align*}
\int_{\mathbb{R}^d} |p_T(x) - \varphi_T(x)| \rmd x &\leqslant \int_{\mathbb{R}^d} p_0(z) \left( \int_{\mathbb{R}^d} \left| \varphi_T(x - z) - \varphi_T(x) \right| \rmd x \right) \rmd z.
\end{align*}
By the definition of TV distance, it holds that
$$
\int_{\mathbb{R}^d} |\varphi_T(x - z) - \varphi_T(x)| \rmd x = 2\operatorname{TV}(\varphi_T(\cdot - z) \| \varphi_T).
$$
Substituting this into the previous inequality provides us with
\begin{align*}
\int_{\mathbb{R}^d} |p_T(x) - \varphi_T(x)| \rmd x &\leqslant 2 \int_{\mathbb{R}^d} p_0(z) \operatorname{TV}(\varphi_T(\cdot - z), \varphi_T) \rmd z.
\end{align*}
Dividing both sides by the factor $2$ gives
\begin{align*}
\operatorname{TV}(p_T, \varphi_T) \leqslant \int_{\mathbb{R}^d} p_0(z) \operatorname{TV}(\varphi_T(\cdot - z), \varphi_T) \rmd z. \tag{1}
\end{align*}
By Pinsker's inequality, we have
\begin{align*}
\operatorname{TV}(\varphi_T(\cdot - z), \varphi_T) \leqslant\sqrt{\frac{1}{2}\operatorname{KL}(\mathcal N(z,TI_d),\mathcal N(0,TI_d))} =\frac{\|z\|}{2 \sqrt{T}}.
\end{align*}
Therefore,
\begin{align*}
\operatorname{TV}(p_T, \varphi_T) &\leqslant \int_{\mathbb{R}^d} p_0(z) \cdot \frac{\|z\|}{2 \sqrt{T}} \rmd z \\
&= \frac{1}{2\sqrt{T}} \int_{\mathbb{R}^d} \|z\| p_0(z) \rmd z \\
&= \frac{M_1}{2\sqrt{T}}.
\end{align*}
Since $M_1 < \infty$ by assumption, this completes the proof.
\end{proof}

{

\subsubsection{Proof of Lemmas for Exponential Decay}
\begin{proof}[Proof of Lemma~\ref{lem:pt_expdecay}]
    \begin{align*}
        p_t(x) = \int_{\mathbb{R}^d} p_0(y) \varphi_t(x-y) \,\mathrm{d}y.
    \end{align*}
    To analyze the decay behavior for large $\|x\|$, we split the integration domain $\mathbb{R}^d$ into two regions. Let $R = \|x\|/2$. We define the regions $D_1 = \{y : \|x-y\| \geqslant R\}$ and $D_2 = \{y : \|x-y\| < R\}$.
    
    We first consider the integral over the first region $D_1$. In this area, the distance between $x$ and $y$ satisfies $\|x-y\| \geqslant \|x\|/2$. Consequently, the Gaussian kernel $\varphi_t(x-y)$ decays rapidly. We can derive the upper bound for the kernel
    \begin{align*}
        \varphi_t(x-y) = \frac{1}{(2\pi t)^{d/2}} \exp\left(-\frac{\|x-y\|^2}{2t}\right) \leqslant \frac{1}{(2\pi t)^{d/2}} \exp\left(-\frac{\|x\|^2}{8t}\right).
    \end{align*}
    Since $\int_{D_1} p_0(y) \mathrm{d}y \leqslant 1$, we then have
    \begin{align*}
        I_1 = \int_{D_1} p_0(y) \varphi_t(x-y) \,\mathrm{d}y \leqslant \dfrac{1}{(2\pi t)^{d/2}} \exp\left(-\frac{\|x\|^2}{8t}\right).
    \end{align*}
    
    Next, we turn to the second region $D_2$. In this area, the noise is small, satisfying $\|x-y\| < \|x\|/2$. By the reverse triangle inequality, the magnitude of $y$ is bounded from below
    \begin{align*}
        \|y\| = \|x - (x-y)\| \geqslant \|x\| - \|x-y\| > \frac{\|x\|}{2}.
    \end{align*}
    Using Assumption~\ref{asm:heavytail} on the initial density $p_0$, for any $y \in D_2$, we have
    \begin{align*}
        p_0(y) \leqslant C_1 \exp(-c_1 \|y\|^\gamma) \leqslant C_1 \exp\left(-c_1 \frac{\|x\|^\gamma}{2^\gamma}\right).
    \end{align*}
    Similar to the previous step, we bound the contribution from this region by taking the supremum of $p_0$ over $D_2$ outside the integral. Since $\int_{D_2} \varphi_t(x-y) \mathrm{d}y \leqslant 1$, it holds that
    \begin{align*}
        I_2 = \int_{D_2} p_0(y) \varphi_t(x-y) \,\mathrm{d}y \leqslant C_1 \exp\left(-2^{-\gamma}c_1 \|x\|^\gamma\right)\,.
    \end{align*}
    Combining the two bounds, the total density satisfies:
    \begin{align*}
        p_t(x) = I_1 + I_2 \leqslant \dfrac{1}{(2\pi t)^{d/2}} \exp\left(-\frac{\|x\|^2}{8t}\right) + C_1 \exp\left(-2^{-\gamma}c_1\|x\|^\gamma\right).
    \end{align*}
    For sufficiently large $\|x\|$, the sum is dominated by the term with the slower decay rate. If $\gamma < 2$, the term $\exp(-2^{-\gamma}c_1 \|x\|^\gamma)$ dominates. If $\gamma \geqslant 2$, the Gaussian term dominates or decays at a comparable rate. Thus, we conclude that $p_t(x) \leqslant \mathscr C_\gamma(t) \exp\left(-c_\gamma(t) \min(\|x\|^2, \|x\|^\gamma)\right)$, where $\mathscr C_\gamma(t)=\frac{1}{(2\pi t)^{d/2}}+C_1$ and $c_\gamma(t)=\min\left(\frac{1}{8t},\frac{c_1}{2^\gamma}\right)$.
\end{proof}

\begin{proof}[Proof of Lemma~\ref{lem:G1_expdecay}]
    The condition $\mathscr C_\gamma(t) \exp(-c_\gamma(t) \|x\|^\gamma) \geqslant c_\alpha\rho_n$ implies 
    $$\|x\|^\gamma \leqslant \frac{1}{c_\gamma(t)} \log(\mathscr C_\gamma(t)/c_\alpha\rho_n)\,.$$ 
    Thus, the radius $R_n \asymp (\log(1/\rho_n))^{1/\gamma}$. The volume scales as $R_n^d \asymp (\log(1/\rho_n))^{d/\gamma}$. This logarithmic growth is negligible compared to the polynomial factors in $n$.
    Thus, we obtain
    \begin{align*}
        \int_{\G_1}\|s_t(x)\|^2\rmd x&\lesssim mt^{-1}\left(\int_{\G_1}p_t(x)^{1-m'}\rmd x\right)^{1/m'}\\
        &\leqslant mt^{-1}\left[|\G_1|(c_\alpha\rho_n)^{1-m'}\right]^{1/m'}\\
        &\leqslant mt^{-1}\left[\left(\log(\mathscr C_\gamma(t)/c_\alpha\rho_n)/c_\gamma(t)\right)^{d/\gamma}(c_\alpha\rho_n)^{1-m'}\right]^{1/m'}\\
        &\lesssim \text{polylog}(nt)\cdot mt^{-1}c_\gamma(t)^{d/(\gamma m')}\rho_n^{-1/m}
    \end{align*}
    as desired.
\end{proof}

\begin{proof}[Proof of Lemma~\ref{lem:tail_prob_exp}]
Let $k = \min(\gamma, 2)$. For sufficiently large $\|x\|$, the tail behaves as $\exp(-c_\gamma(t)\|x\|^k)$. 
Then the tail probability is bounded as follows
\begin{align*}
    \mathbb{P}(p_t(X) \leqslant c_\alpha\rho_n) 
    &\leqslant \mathbb{P}(p_t(X) \leqslant c_\alpha\rho_n, \|X\| < R_n) + \mathbb{P}(\|X\| \geqslant R_n) \\
    &= \int_{\|x\| < R_n} p_t(x) \mathbbm{1}\{p_t(x) \leqslant c_\alpha\rho_n\} \,\mathrm{d}x + \int_{\|x\| \geqslant R_n} \mathscr C_\gamma(t) \exp(-c_\gamma(t) \|x\|^k) \,\mathrm{d}x \\
    &\leqslant c_\alpha\rho_n \text{Vol}(\mathbb B_{R_n}) + \mathscr C_\gamma(t) \frac{2\pi^{d/2}}{\Gamma(d/2)} \int_{R_n}^\infty r^{d-1} \exp(-c_\gamma(t) r^k) \,\mathrm{d}r \\
    &\leqslant c_\alpha\rho_n \frac{\pi^{d/2}}{\Gamma(d/2+1)} R_n^d + \mathscr C_\gamma(t) \frac{2\pi^{d/2}}{\Gamma(d/2)} \cdot \frac{R_n^{d-k}}{k c_\gamma(t)} \exp(-c_\gamma(t) R_n^k)\,.
\end{align*}
Using the standard tail asymptotic
\begin{align*}
\int_{R}^\infty r^{d-1}e^{-cr^k}\mathrm{d}r \sim \frac{R^{d-k}}{ck}e^{-cR^k} \qquad \text{as }R\to\infty
\end{align*}
we obtain
\begin{align*}
    \P(p_t(X)\leqslant c_\alpha\rho_n)&= c_\alpha\rho_n C_1 R_n^d +\mathscr C_\gamma(t) \exp(-c_\gamma(t) R_n^k) \cdot C_2 R_n^{d-k} \\
    &\leqslant C' \rho_n (R_n^d + R_n^{d-k}) \\
    &\lesssim \rho_n R_n^d\\
    &\asymp \rho_n (\log(1/\rho_n))^{d/k}\,.
\end{align*}
where in the last step we use the relation derived from the tail decay
\begin{align*}
\exp(-c_\gamma(t) R_n^k) \asymp \rho_n 
\end{align*}
which yields
\begin{align*}
R_n \asymp (\log(1/\rho_n))^{1/k}\,.
\end{align*}
This completes the proof.

\end{proof}
}

\subsection{Alternative Low-Noise Strategy for Polynomially Decaying Targets}
\label{sec:decoupled_score}

As discussed in Section~\ref{sec:proof_poly_MSEst}, in the low-noise regime (small diffusion time $t$), the naive estimator based on the intrinsic diffusion scale $\sqrt{t}$ suffers from a variance blow-up in heavy-tailed regions due to sample sparsity.
To mitigate this issue without resorting to early stopping, we introduce a decoupled kernel score estimator.
Throughout this section, we assume that $p_0$ satisfies Assumption~\ref{asm:p0smooth} with $\beta>2$, and for pointwise error control, we further impose a Hölder continuity condition on $p_0$.

\begin{assumption}[Hölder Regularity]\label{asm:holder}
    The target density $p_0$ is bounded and belongs to the Hölder class $\Sigma(\beta, L)$ with $\beta > 2$. Let $\ell = \lfloor \beta \rfloor$ denote the integer part of $\beta$. 
    Assume $p_0$ is $\ell$-times continuously differentiable, and its $\ell$-th order partial derivatives are Hölder continuous
    \begin{align*}
        |D^\alpha p_0(x) - D^\alpha p_0(y)| \leqslant L \|x - y\|^{\beta - \ell}, \quad  \forall  x, y \in \mathbb{R}^d, \, |\alpha| = \ell,
    \end{align*}
    where $D^\alpha$ denotes the mixed partial derivative associated with the multi-index $\alpha$, and $L>0$ is a universal constant independent of $x$ and $y$.
\end{assumption}
We note that the Hölder constant $L$ is preserved uniformly for all $t$, since 
\begin{align*}
    |D^\alpha p_t(x)-D^\alpha p_t(y)|&=|D^\alpha (p_0*\varphi_t)(x)-D^\alpha (p_0*\varphi_t)(y)|\\
    &=\int_{\R^d}[(D^\alpha p_0)(x-z)-(D^\alpha p_0)(y-z)]\varphi_t(z)\rmd z\\
    &\leqslant\int_{\R^d}L\|x-y\|^{\beta-\ell}\varphi_t(z)\rmd z\\
    &=L\|x-y\|^{\beta-\ell} \,.
\end{align*}

\noindent\textbf{Score estimator construction.} 
Let $K: \mathbb{R}^d \to \mathbb{R}$ be a compactly supported kernel of order $\ell = \lfloor \beta \rfloor$, supported on $[-1,1]^d$, satisfying
$\int_{\mathbb{R}^d} K(u) \,\rmd u = 1$,
$\int_{\mathbb{R}^d} u^\alpha K(u) \,\rmd u = 0$ for all multi-indices $\alpha$ with $1 \leqslant |\alpha| \leqslant \ell$,
and $K, \nabla K \in L^2(\mathbb{R}^d)$.

The kernel density estimator for $p_0$ with bandwidth $h > 0$ is defined via
\begin{align}\label{eq:kde_p0}
    \hat{f}_h(y) = \frac{1}{n h^d} \sum_{i=1}^n K\left(\frac{y - X_i}{h}\right), \quad \text{where } X_i \overset{\text{i.i.d.}}{\sim} p_0.
\end{align}
The density estimator is defined via convolution with the Gaussian kernel $\varphi_t$
\begin{align}\label{eq:smooth_estimator}
    \hat{p}_{t}(x) := (\hat{f}_h * \varphi_t)(x) = \int_{\mathbb{R}^d} \hat{f}_h(y) \varphi_t(x-y) \,\rmd y.
\end{align}
We stabilize the score estimator using a unified threshold $\rho_n = C_\rho h^\beta$, 
where $C_\rho>0$ is chosen sufficiently large.
Specifically,
\begin{align}\label{eq:score_estimator_low}
    \hat{s}_t(x) := \frac{\nabla \hat{p}_{t}(x)}{\hat{p}_{t}(x)} \mathbbm{1}_{\{\hat{p}_{t}(x) \geqslant \rho_n\}}.
\end{align}

We first establish global MISE bounds using Fourier analysis, then derive pointwise error bounds for the proposed score estimator.

\begin{proposition}[Global MISE Bounds]\label{prop:mise_fourier}
    Suppose Assumption~\ref{asm:p0smooth} holds. Assume the kernel $K$ satisfies the Fourier-domain approximation property 
    $$
    |\mathcal{F}[K](v) - 1| \leqslant C_K \min(\|v\|^{\beta},1),\qquad \text{for small } \|v\|\,,
    $$
    for some $C_K$ depends on $K$ and $d$.
    Then, for any $t > 0$, the decoupled estimator satisfies 
    \begin{align}
        \|\nabla \hat{p}_{t} - \nabla p_t\|_{\Ltwo}^2 &\leqslant (C_K^2+1) L^2 h^{2(\beta-1)} + \frac{\|\nabla K\|_{L^2}^2}{n h^{d+2}}\, \label{eq:grad_mise} \\
        \|\hat{p}_{t} - p_t\|_{\Ltwo}^2 &\leqslant (C_K^2+1) L^2 h^{2\beta} + \frac{\|K\|_{L^2}^2}{n h^d}\,, \label{eq:dens_mise}
    \end{align}
    In particular, these bounds are uniform in $t.$
\end{proposition}

\begin{proof}
   We invoke Plancherel's theorem under the Fourier transform convention  $$\mathcal{F}[g](\omega) = \int_{\mathbb{R}^d} g(x) e^{-i\omega \cdot x} \rmd x$$.
   We then decompose the MISE into the integrated squared bias and the integrated variance. We focus on $\nabla \hat{p}_t$. The density case follows by the same argument.

We begin with the integrated squared bias. Observe that
\[
\mathbb{E}[\nabla \hat{p}_{t}] = \nabla(p_0 * K_h) * \varphi_t\,,
\]
so the pointwise bias can be written as
\[
\text{Bias}(x) = \nabla\left((p_0 * K_h) - p_0\right) * \varphi_t(x)\,.
\]
Taking Fourier transforms gives
\[
\mathcal{F}[\text{Bias}](\omega) = i\omega \left(\mathcal{F}[K](h\omega) - 1\right) \mathcal{F}[p_0](\omega) e^{-t\|\omega\|^2/2}.
\]
    By Plancherel's theorem, the squared $\Ltwo$-norm of the bias is therefore
    $$\|\text{Bias}(\nabla \hat{p}_t)\|_{\Ltwo}^2=\|\text{Bias}(\nabla \hat{p}_t)\|_{L^2}^2
    = \frac{1}{(2\pi)^d} \int_{\mathbb{R}^d} \|\omega\|^2 |\mathcal{F}[K](h\omega) - 1|^2 |\mathcal{F}[p_0](\omega)|^2 e^{-t\|\omega\|^2} \,\rmd \omega.$$
    Using the bound $e^{-t\|\omega\|^2} \leqslant 1$, the kernel's approximation property 
    $$|\mathcal{F}[K](z) - 1| \leqslant C_K \|z\|^\beta \qquad \text{when } \|z\| < 1$$ 
    and the bound
    $$|\mathcal{F}[K](z) - 1| \leqslant 2 \qquad \text{when } \|z\| \geqslant 1,$$
    we split the integral over the regions  $\|h\omega\|<1$ and $\|h\omega\|\geqslant 1$.
    This yields
    \begin{align*}
        \|\text{Bias}(\nabla \hat{p}_t)\|_{\Ltwo}^2
        &\leqslant \frac{1}{(2\pi)^d} \left(C_K^2\int_{\|h\omega\|< 1} \|\omega\|^2 \|h\omega\|^{2\beta} |\mathcal{F}[p_0](\omega)|^2 \,\rmd \omega+\int_{\|h\omega\|\geqslant 1}\|\omega\|^2 \cdot 4 |\mathcal{F}[p_0](\omega)|^2\,\rmd \omega\right) \\
        &\leqslant \frac{h^{2(\beta-1)}}{(2\pi)^d} \left(C_K^2\int_{\|h\omega\|< 1} \|\omega\|^{2\beta} |\mathcal{F}[p_0](\omega)|^2 \,\rmd \omega+\int_{\|h\omega\|\geqslant 1}\|\omega\|^{2\beta}|\mathcal{F}[p_0](\omega)|^2\,\rmd \omega\right) \\
        &\leqslant \frac{(C_K^2+1) L^2 h^{2(\beta-1)}}{(2\pi)^d} \int_{\mathbb{R}^d} \|\omega\|^{2\beta} |\mathcal{F}[p_0](\omega)|^2 \,\rmd \omega \\
        &\leqslant (C_K^2+1) L^2 h^{2(\beta-1)},
    \end{align*}
    where the last step follows from the fact that $\int_{\mathbb{R}^d} \|\omega\|^{2\beta} |\mathcal{F}[p_0](\omega)|^2 \rmd \omega \leqslant (2\pi)^d L^2$.

    Next, we analyze the integrated variance.
    Define $\Psi_{h,t} = \nabla K_h * \varphi_t$, then the integrated variance of $\nabla \hat{p}_t$ is bounded by $\frac{1}{n} \|\Psi_{h,t}\|_{\Ltwo}^2$. By Young's convolution inequality and $\|\varphi_t\|_{L^1}=1$, it holds that 
    $$\|\Psi_{h,t}\|_{\Ltwo}=\|\Psi_{h,t}\|_{L^2} \leqslant \|\nabla K_h\|_{L^2} \|\varphi_t\|_{L^1} = \|\nabla K_h\|_{L^2}.$$
    By the scaling of the kernel gradient, it holds that
    $$\nabla K_h(x) = h^{-(d+1)} \nabla K(x/h)\,,$$
    and therefore
    $$\|\nabla K_h\|_{L^2}^2 = \int_{\mathbb{R}^d} h^{-2(d+1)} |\nabla K(x/h)|^2 \rmd x = h^{-(d+2)} \|\nabla K\|_{L^2}^2.$$
    Substituting this identity yields the integrated variance term  $\frac{\|\nabla K\|_{L^2}^2}{n h^{d+2}}$.
    
\end{proof}

\begin{proposition}[Uniform Pointwise Bias]\label{prop:pointwise_bias}
    Under Assumption~\ref{asm:holder}, for any $t > 0$ and $x \in \mathbb{R}^d$, the pointwise bias of the decoupled estimators satisfies:
    \begin{align}
        \left| \mathbb{E}[\hat{p}_t(x)] - p_t(x) \right| &\leqslant C_{B,0} h^\beta, \label{eq:bias_density}\\
        \left\| \mathbb{E}[\nabla \hat{p}_t(x)] - \nabla p_t(x) \right\| &\leqslant C_{B,1} h^{\beta-1}, \label{eq:bias_gradient}
    \end{align}
    where $C_{B,0} = L \cdot \frac{(2d)^\ell}{\ell!} \int_{\mathbb{R}^d} \|u\|^\beta |K(u)|\rmd u$ and $C_{B,1} = L \cdot \frac{(2d)^{\ell-1}}{(\ell-1)!} \int_{\mathbb{R}^d} \|u\|^{\beta-1} \|\nabla K(u)\|\rmd u$, depending only on $L, d, \ell$ and the kernel.
\end{proposition}

\begin{proof}
We begin by studying the bias of the density estimator $\hat{p}_t$.
The bias admits the representation
    \begin{align*}
    \text{Bias}(\hat{p}_t)(x) = \left( (p_0 * K_h) - p_0 \right) * \varphi_t(x).
    \end{align*}
    Since $\varphi_t$ is a probability density, convolution with $\varphi_t$ is an $L^\infty$-contraction. Consequently,
    $$\sup_x |\text{Bias}(\hat{p}_t)(x)| \leqslant \sup_y |(p_0 * K_h)(y) - p_0(y)|\,.$$

    For any fixed $y$, expanding $p_0(y-hu)$ around $y$ via the multivariate Taylor formula with an integral remainder gives
    \begin{align*}
    p_0(y-hu) &= p_0(y) + \sum_{1\leqslant|\alpha|\leqslant\ell} \frac{(-h)^{|\alpha|}u^\alpha}{\alpha!}D^\alpha p_0(y) \\
    &\quad+ \frac{(-h)^\ell}{\ell!} \sum_{|\alpha|=\ell} \alpha! \int_0^1 (1-s)^{\ell-1} [D^\alpha p_0(y-shu)-D^\alpha p_0(y)]\rmd s \, u^\alpha.
    \end{align*}
    The polynomial terms vanish after integration because the kernel has the required vanishing moments. Under Assumption~\ref{asm:holder}, the remainder term satisfies
    $$|R_\ell(y,-hu)| \leqslant L \cdot \frac{(2d)^\ell}{\ell!} \|hu\|^\beta.$$ Consequently,
    $$|(p_0 * K_h)(y)-p_0(y)| \leqslant \int_{\mathbb{R}^d} |K(u)| \cdot L \cdot \frac{(2d)^\ell}{\ell!} h^\beta \|u\|^\beta \rmd u =: C_{B,0}h^\beta.$$

    Next, we analyze the bias of the gradient estimator $\nabla \hat{p}_t$. It can be written as
    $\nabla \hat{p}_t$:
    $$\text{Bias}(\nabla \hat{p}_t)(x) = (\nabla(p_0*K_h)-\nabla p_0)*\varphi_t(x).$$
    Using the identity $\nabla(p_0*K_h)=(\nabla p_0)*K_h$, we can repeat the same Taylor expansion argument with $\nabla p_0$ in place of $p_0$.
    Since $\nabla p_0 \in \Sigma(\beta-1,L)$ when $\beta>1$, the bound follows exactly as before, with $\beta$ replaced by $\beta-1$ and $K$ replaced by $\nabla K$.

\end{proof}

\begin{proposition}[Uniform Pointwise Variance]\label{prop:pointwise_variance}
    For any $t > 0$ and $x \in \mathbb{R}^d$, the pointwise variance of the decoupled estimators is bounded by the local density $p_t^*(x) = \sup_{\|y-x\|\leqslant h} p_t(y)$ via
    \begin{align}
        \operatorname{Var}(\hat{p}_t(x)) \leqslant \frac{\|K\|_{L^2}^2 \, p_t^*(x)}{n h^d}\,, \quad
        \operatorname{Var}(\nabla \hat{p}_t(x)) \leqslant \frac{\|\nabla K\|_{L^2}^2 \, p_t^*(x)}{n h^{d+2}}\,,
    \end{align}
    where the bounds are independent of $t$.
\end{proposition}

\begin{proof}
We first bound the variance of the density estimator $\hat{p}_t$.
Note that 
$$\hat{p}_t(x) = \frac{1}{n}\sum_{i=1}^n Z_i,
\qquad
Z_i=(K_h*\varphi_t)(x-X_i) \,.
$$
Since $Z_i$ are i.i.d, we have
$$
\operatorname{Var}(\hat{p}_t(x)) \leqslant \frac{1}{n}\mathbb{E}[Z_1^2]\,.
$$
By Jensen’s inequality applied to convolution with the probability density $\varphi_t$
\begin{align*}
    (K_h*\varphi_t)^2(z) \leqslant (K_h^2 * \varphi_t)(z).
\end{align*}
which implies 
$$
\mathbb{E}[Z_1^2] \leqslant (K_h^2 * p_t)(x)\,.
$$
If $K$ is compactly supported, then 
$K_h^2(u)=h^{-2d}K^2(u/h)$ has support in $[-h,h]^d$, and a change of variables yields
$$(K_h^2 * p_t)(x) = \frac{1}{h^d}\int_{\mathbb{R}^d} K^2(u)p_t(x-hu)\rmd u \leqslant \frac{p_t^*(x)}{h^d}\int_{\mathbb{R}^d} K^2(u) \rmd u = \frac{\|K\|_{L^2}^2 p_t^*(x)}{h^d}.$$
where $p_t^*(x):=\sup_{\|u\|\leqslant 1}p_t(x-hu)$. Dividing by $n$ gives the stated variance bound for $\hat p_t(x).$

    Then, we bound the variance of the gradient estimator $\nabla \hat{p}_t$. Note that
    $$\nabla \hat{p}_t(x) = \frac{1}{n}\sum_{i=1}^n W_i,\qquad W_i=(\nabla K_h*\varphi_t)(x-X_i)\,.$$
    Using the scaling relation
$$\nabla K_h(u)=h^{-(d+1)}\nabla K(u/h)\,,$$
we see that the gradient introduces an extra factor $h^{-1}$ compared with $K_h$. Applying the same Jensen inequality and the same compact support argument as in the density case yields an analogous bound, with $\| K\|_{L^2}^2$ replaced by $\|\nabla K\|_{L^2}^2$ and $h^d$ replaced by $h^{d+2}.$
   
\end{proof}

\begin{lemma}[High-Probability Concentration]\label{lem:concentration_pointwise}
    Let $\hat{p}_t(x) = (\hat{f}_h * \varphi_t)(x)$ with bandwidth $h \asymp n^{-1/(2\beta+d)}$ and threshold $\rho_n = C_\rho h^\beta$, where $C_\rho>\max(2C_{B,0},L\cdot\frac{(2d)^\ell}{2\ell!})$. For any $t > 0$ and $x \in \mathbb{R}^d$, with probability at least $1 - n^{-\alpha}$:
    \begin{align}
        |\hat{p}_t(x) - p_t(x)| \leqslant C_{B,0}h^\beta + C_\alpha \left(\sqrt{p_t^*(x) \frac{\log n}{n h^d}}+\dfrac{\log n}{nh^d}\right)\,,
    \end{align}
    where $C_\alpha = \max\left(\sqrt{8\alpha\|K\|_{L^2}^2}, \frac{8\alpha\|K\|_\infty}{3}\right)$. On the set $G_1 = \{x : p_t(x) \geqslant 2\rho_n\}$, we further have $\hat{p}_t(x) \geqslant \frac{1}{2}p_t(x)$ for sufficiently large $n$.
\end{lemma}

\begin{proof}
    We first apply Bernstein's inequality to $Z_i=(K_h*\varphi_t)(x-X_i)$, with $|Z_i|\leqslant h^{-d}\|K\|_\infty=:M$ and then $\text{Var}(Z_i)\leqslant \frac{\|K\|_{L^2}^2 p_t^*(x)}{h^d}$. Setting $\lambda = \max\left(\sqrt{4\alpha\sigma^2\log(2n)}, \frac{4\alpha M\log(2n)}{3n}\right)$.
     This ensures $$\mathbb{P}(|\hat{p}_t-\mathbb{E}\hat{p}_t|>\lambda)\leqslant n^{-\alpha}.$$

    The total error decomposes into bias and stochastic error
    \begin{align*}
    |\hat{p}_t-p_t| \leqslant |\mathbb{E}\hat{p}_t-p_t| + |\hat{p}_t-\mathbb{E}\hat{p}_t| \leqslant C_{B,0}h^\beta + C_\alpha\left(\sqrt{p_t^*(x)\frac{\log n}{n h^d}}+\dfrac{\log n}{nh^d}\right)\,,
    \end{align*}
    where $C_\alpha = \max\left(\sqrt{8\alpha\|K\|_{L^2}^2}, \frac{8\alpha\|K\|_\infty}{3}\right)$. 
    By \Holder regularity of $p_t$, we have
    \begin{align*}
        |p_t(x+hu)-p_t(x)|\leqslant L\cdot\dfrac{(2d)^\ell}{\ell!}\|hu\|^\beta \,.
    \end{align*}
    Therefore, for $x\in G_1$, it holds that
    \begin{align*}
        p_t^*(x)\leqslant p_t(x)+L\cdot\dfrac{(2d)^\ell}{\ell!}h^\beta\leqslant 2p_t(x)\,.
    \end{align*}
    Thus, the stochastic term satisfies $$\sqrt{p_t^*(x)\frac{\log n}{n h^d}} \leqslant \sqrt{2p_t(x) \cdot \frac{\log n}{C_\rho^2}\rho_n^2} = \sqrt{\frac{h^\beta\log n}{C_\rho}}p_t(x).$$
    We also have $$\frac{\log n}{nh^d}\leqslant \frac{h^\beta\log n}{C_\rho}\rho_n.$$
    Therefore,
    \begin{align*}
    |\hat p_t-p_t|\leqslant \frac{1}{2}\rho_n+C_\alpha\left(\sqrt{\dfrac{h^\beta\log n}{C_\rho}}+\dfrac{h^\beta\log n}{2C_\rho}\right)p_t(x)\leqslant \frac{1}{2}p_t(x)
    \end{align*}
    for large $n$, implying $\hat{p}_t(x)\geqslant\frac{1}{2}p_t(x)$.
\end{proof}

\begin{assumption}[Score Regularity]\label{asm:score_moment}
    The initial distribution $p_0$ satisfies $\mathbb{E}_{p_0}[\|s_0(x)\|^{2k}] < \infty$ for some $k \geqslant 1$.
\end{assumption}

\begin{lemma}[Monotonicity of Score Moments]\label{lem:score_monotonicity}
    Suppose the initial score function $s_0 = \nabla \log p_0$ satisfies $\mathbb{E}_{p_0}[\|s_0(x)\|^{2k}] < \infty$ for some $k \geqslant 1$. Then, for any $t > 0$, the moments of the diffused score $s_t$ are non-increasing with respect to $t$
    \begin{align}
        \mathbb{E}_{p_t}[\|s_t(x)\|^{2k}] \leqslant \mathbb{E}_{p_0}[\|s_0(x)\|^{2k}].
    \end{align}
\end{lemma}

\begin{proof}
    We utilize the interpretation of the diffused score as a conditional expectation (often referred to as Tweedie's formula). The score function can be expressed as
    \begin{align*}
        s_t(x) = \nabla \log p_t(x) = \frac{\nabla p_t(x)}{p_t(x)} = \frac{\int_{\mathbb{R}^d} \nabla p_0(y) \varphi_t(x-y) \,\rmd y}{p_t(x)}.
    \end{align*}
    Using the identity $\nabla p_0(y) = s_0(y) p_0(y)$, we have
    \begin{align*}
        s_t(x) = \int_{\mathbb{R}^d} s_0(y) \frac{p_0(y) \varphi_t(x-y)}{p_t(x)} \,\rmd y.
    \end{align*}
    Notice that $w(y|x) := \frac{p_0(y) \varphi_t(x-y)}{p_t(x)}$ corresponds exactly to the posterior density $p(y|x)$ (the conditional distribution of the clean data $y$ given the noisy observation $x$). Thus, $s_t(x)$ is the posterior expectation of the initial score
    \begin{align*}
        s_t(x) = \mathbb{E}_{Y|X=x}[s_0(Y)].
    \end{align*}
    Consider the convex function $\Phi(u) = \|u\|^{2k}$ (for $k \geqslant 1$). By Jensen's inequality for conditional expectations, it holds that
    \begin{align*}
        \|s_t(x)\|^{2k} = \left\| \mathbb{E}_{Y|X=x}[s_0(Y)] \right\|^{2k} \leqslant \mathbb{E}_{Y|X=x}\left[ \|s_0(Y)\|^{2k} \right].
    \end{align*}
    Now, we compute the global expectation with respect to the marginal density $p_t(x)$
    \begin{align*}
        \mathbb{E}_{p_t}[\|s_t(x)\|^{2k}] &= \int_{\mathbb{R}^d} \|s_t(x)\|^{2k} p_t(x) \,\rmd x \\
        &\leqslant \int_{\mathbb{R}^d} \left( \int_{\mathbb{R}^d} \|s_0(y)\|^{2k} \frac{p_0(y) \varphi_t(x-y)}{p_t(x)} \,\rmd y \right) p_t(x) \,\rmd x \\
        &= \int_{\mathbb{R}^d} \int_{\mathbb{R}^d} \|s_0(y)\|^{2k} p_0(y) \varphi_t(x-y) \,\rmd y \,\rmd x.
    \end{align*}
    By Fubini's theorem, we have
    \begin{align*}
        \mathbb{E}_{p_t}[\|s_t(x)\|^{2k}] &\leqslant \int_{\mathbb{R}^d} \|s_0(y)\|^{2k} p_0(y) \left( \int_{\mathbb{R}^d} \varphi_t(x-y) \,\rmd x \right) \rmd y.
    \end{align*}
    Since $\varphi_t$ is a probability density, and $\int \varphi_t(x-y) \,\rmd x = 1$, it follows that
    \begin{align*}
        \mathbb{E}_{p_t}[\|s_t(x)\|^{2k}] \leqslant \int_{\mathbb{R}^d} \|s_0(y)\|^{2k} p_0(y) \,\rmd y = \mathbb{E}_{p_0}[\|s_0(x)\|^{2k}]
    \end{align*}
    as desired.
    
\end{proof}

We are now ready to establish the MSE for the proposed score estimator. 
\begin{theorem}[MSE for $\hat s_t$]\label{thm:score_error_low}
    Suppose Assumptions~\ref{asm:p0tail}, \ref{asm:p0smooth} with $\beta>2$,~\ref{asm:holder} and~\ref{asm:score_moment} hold.
    Set $h \asymp n^{-\frac{1}{2\beta+d}}$ and $\rho_n=C_\rho h^\beta\asymp n^{-\frac{\beta}{2\beta+d}}$. For any $t > 0$, the score estimation error satisfies
    \begin{align}
        \mathbb{E}\left[\int_{\mathbb{R}^d} \|\hat{s}_t(x) - s_t(x)\|^2 p_t(x) \,\rmd x\right] \lesssim n^{-\frac{\beta-2}{2\beta+d}}+n^{-\frac{\beta(\gamma+1)}{k'(2\beta+d)(d+\gamma+1)}}\,,
    \end{align}
    where $\frac{1}{k}+\frac{1}{k'}=1$.
\end{theorem}

\begin{proof}
    We analyze the score estimation error by decomposing the weighted mean integrated squared error $\mathcal{R} = \mathbb{E}\left[\int_{\mathbb{R}^d} \|\hat{s}_t(x) - s_t(x)\|^2 p_t(x) \,\rmd x\right]$ based on the behavior of the density estimator $\hat{p}_t$. We split $\mathcal{R}$ into two distinct components
    \begin{align*}
        \mathcal{R} = E_{\text{trunc}} + E_{\text{est}},
    \end{align*}
    where $E_{\text{trunc}} = \mathbb{E}\int_{\{\hat{p}_t < \rho_n\}} \|s_t\|^2 p_t \,\rmd x$ and $E_{\text{est}} = \mathbb{E}\int_{\{\hat{p}_t \geqslant \rho_n\}} \left\| \frac{\nabla \hat{p}_t}{\hat{p}_t} - s_t \right\|^2 p_t \,\rmd x$. The component $E_{\text{trunc}}$ corresponds to the case where $\hat{p}_t$ is too small to define a valid score estimate. The component $E_{\text{est}}$ corresponds to the case where $\hat{p}_t$ is sufficiently large, allowing direct computation of the score estimate.

    We introduce two auxiliary regions to refine the error analysis. Define the theoretical high-density region $G_1 = \{x : p_t(x) \geqslant 2\rho_n\}$ and the tail region $G_2 = \{x : p_t(x) < 2\rho_n\}$. Recall the high-probability event $A_\alpha$ from Lemma~\ref{lem:concentration_pointwise}. On this event, the uniform bound $|\hat{p}_t(x) - p_t(x)| \leqslant \frac{1}{2}p_t(x)$ holds for all $x\in G_1$. The probability of the complement event $A_\alpha^c$ is $\mathbb{P}(A_\alpha^c) \leqslant n^{-\alpha}$, which is negligible for large $n$.

    We now analyze $E_{\text{trunc}}$. To bound this truncation error, we split the integral over the auxiliary regions $G_1$ and $G_2$. This split separates the main mass of the density from its tail, simplifying the bound:
    \begin{align*}
        E_{\text{trunc}} = \mathbb{E}\left[ \int_{G_2 \cap \{\hat{p}_t < \rho_n\}} \|s_t\|^2 p_t \,\rmd x \right] + \mathbb{E}\left[ \int_{G_1 \cap \{\hat{p}_t < \rho_n\}} \|s_t\|^2 p_t \,\rmd x \right].
    \end{align*}

    We first consider the contribution from the tail region. We call this part the tail truncation error. The first term in the split is bounded by the total mass of $\|s_t\|^2 p_t$ over $G_2$. We apply Hölder's inequality with conjugate exponents $k$ and $k'$, where $\frac{1}{k}+\frac{1}{k'}=1$. Assumption~\ref{asm:score_moment} ensures $\mathbb{E}_{p_t}[\|s_t\|^{2k}] < \infty$, so the inequality gives:
    \begin{align*}
        \int_{G_2} \|s_t\|^2 p_t \,\rmd x \leqslant \left(\int_{\mathbb{R}^d} \|s_t\|^{2k} p_t(x)\,\rmd x\right)^{1/k} \left( \int_{G_2} p_t(x)\,\rmd x \right)^{1/k'} \lesssim (\mathbb{P}(p_t(X) < 2\rho_n))^{1/k'}.
    \end{align*}
    Following the same argument as in Eq.~\eqref{eq:G2Prob}, the tail probability $\mathbb{P}(p_t(X) < 2\rho_n)$ scales as $\rho_n^{\frac{\gamma+1}{d+\gamma+1}}$. Substituting the scaling relation $\rho_n \asymp h^\beta$, this term further scales as $h^{\frac{\beta(\gamma+1)}{k'(d+\gamma+1)}}$.

    We next consider the contribution from the high-density region. On the high-probability event $A_\alpha$, any $x \in G_1$ satisfies $\hat{p}_t(x) \geqslant \frac{1}{2}p_t(x)$. Since $p_t(x) \geqslant 2\rho_n$ for $x\in G_1$, we have $\hat{p}_t(x) \geqslant \rho_n$. The set $G_1 \cap \{\hat{p}_t < \rho_n\}$ is therefore empty on $A_\alpha$. The only contribution to this term comes from $A_\alpha^c$, which is $O(n^{-\alpha})$ and negligible. We can absorb this negligible term into the constant factor in the final bound.

    We now analyze $E_{\text{est}}$, the estimation error where $\hat{p}_t \geqslant \rho_n$. The lower bound $\hat{p}_t \geqslant \rho_n$ ensures the denominator in the score estimate is bounded below. This gives $\hat{p}_t^{-1} \leqslant \rho_n^{-1}$, a key bound for subsequent analysis. Using the triangle inequality, we decompose the score error into two components.
    \begin{align*}
        \left\| \frac{\nabla \hat{p}_t}{\hat{p}_t} - \frac{\nabla p_t}{p_t} \right\|^2 p_t \lesssim T_1 + T_2,
    \end{align*}
    where $T_1 = \frac{\|\nabla \hat{p}_t - \nabla p_t\|^2}{\hat{p}_t^2} p_t$ and $T_2 = \frac{|\hat{p}_t - p_t|^2}{\hat{p}_t^2} \|s_t\|^2 p_t$.

    We first derive a uniform bound for the weighting factor $\frac{p_t(x)}{\hat{p}_t^2(x)}$ over all $x$ with $\hat{p}_t(x) \geqslant \rho_n$. This uniform bound simplifies the analysis of both $T_1$ and $T_2$. For $x \in G_1$, the event $A_\alpha$ gives $\hat{p}_t(x) \geqslant \frac{1}{2} p_t(x)$. This leads to $\frac{p_t(x)}{\hat{p}_t^2(x)} \leqslant \frac{4}{p_t(x)}$. Since $p_t(x) \geqslant 2\rho_n$ for $x\in G_1$, we further get $\frac{p_t(x)}{\hat{p}_t^2(x)} \leqslant \frac{2}{\rho_n}$. For $x \in G_2$, the condition $\hat{p}_t(x) \geqslant \rho_n$ and $p_t(x) < 2\rho_n$ directly imply $\frac{p_t(x)}{\hat{p}_t^2(x)} < \frac{2\rho_n}{\rho_n^2} = \frac{2}{\rho_n}$. The weighting factor $\frac{p_t(x)}{\hat{p}_t^2(x)}$ is thus $O(\rho_n^{-1})$ uniformly over the active region of the estimator.

    We analyze the gradient component $T_1$ first. We use the uniform bound on the weighting factor and substitute the MISE bound from Proposition~\ref{prop:mise_fourier}. This substitution connects the gradient error to known results on density estimation:
    \begin{align*}
        \mathbb{E}\left[\int_{\{x:\hat{p}_t(x) \geqslant \rho_n\}} T_1 \,\rmd x \right]
        \lesssim \frac{1}{\rho_n} \|\nabla \hat{p}_t - \nabla p_t\|_{\Ltwo}^2\leqslant \frac{1}{\rho_n}\left[(C_K^2+1)L^2h^{2(\beta-1)}+\frac{\|\nabla K\|_{L^2}^2}{nh^{d+2}}\right].
    \end{align*}
    The scaling relation $\rho_n \asymp h^\beta$ tells us the dominant term in the brackets is $(C_K^2+1)L^2h^{2(\beta-1)}$. Substituting this scaling gives $T_1 \lesssim h^{-\beta} \cdot h^{2\beta-2} = h^{\beta-2}$. Our assumption $\beta > 2$ ensures this term is convergent.

    We then analyze the density component $T_2$. We apply Hölder's inequality again to separate the score moment from the density error term. This separation leverages the boundedness of score moments from Assumption~\ref{asm:score_moment} and Lemma~\ref{lem:score_monotonicity},
    \begin{align*}
        \mathbb{E}\left[\int_{\{x:\,\hat p_t(x)\geqslant \rho_n\}} T_2 \,\rmd x \right]
        &= \mathbb{E}\left[\int_{\{x:\,\hat p_t(x)\geqslant \rho_n\}} \left( \frac{|\hat{p}_t(x) - p_t|^2}{\hat{p}_t^2} \right) \|s_t\|^2 p_t \,\rmd x \right]\\
        &\leqslant \left( \mathbb{E}_{p_t}[\|s_t\|^{2k}] \right)^{1/k} \left( \mathbb{E}\int_{\{x:\,\hat p_t(x)\geqslant \rho_n\}} \left( \frac{|\hat{p}_t - p_t|^2}{\hat{p}_t^2} \right)^{k'} p_t \,\rmd x \right)^{1/k'}.
    \end{align*}
    The first factor is bounded by a constant $M_{2k}^{1/k}$. For the second factor, we use the upper bound of $p_t^*(x) = \sup_{u \in [-h,h]^d} p_t(x+hu)$. The Hölder regularity of $p_t$ gives $p_t^*(x) \leqslant p_t(x) + \frac{L(2d)^\ell}{\ell!}h^\beta$.
    Therefore,
    \begin{align*}
        &\quad\E\int_{\{x:\,\hat p_t(x)\geqslant \rho_n\}}\left(\dfrac{|\hat p_t-p_t|^2}{\hat p_t^2}\right)^{k'}p_t\,\rmd x\\
        &\leqslant \E\int_{\{x:\,\hat p_t(x)\geqslant \rho_n\}}\dfrac{1}{\hat p_t^{2k'-1}}|\hat p_t-p_t|^{2k'}\rmd x\\
        &\leqslant \E\int_{\{x:\,\hat p_t(x)\geqslant \rho_n\}}\dfrac{1}{\hat p_t^{2k'-1}}\left(C_{B,0}^2h^{2\beta}+\dfrac{\|K\|_{L^2}^2p_t^*(x)}{nh^d}\right)^{k'-1}|\hat p_t-p_t|^2\rmd x\\
        &\leqslant \left[\dfrac{1}{\rho_n^{2k'-1}}\left(C_{B,0}^2h^{2\beta}+\frac{\|K\|_{L^2}^2L(2d)^\ell h^\beta}{nh^d}\right)^{k'-1}+\frac{1}{\rho_n^{k'}}\left(\dfrac{\|K\|_{L^2}^2}{nh^d}\right)^{k'-1}\right]\E\int_{\{x:\,\hat p_t(x)\geqslant \rho_n\}}|\hat p_t-p_t|^2\rmd x\\
        &\leqslant \left[\dfrac{1}{\rho_n^{2k'-1}}\left(C_{B,0}^2h^{2\beta}+\frac{\|K\|_{L^2}^2L(2d)^\ell h^\beta}{nh^d}\right)^{k'-1}+\frac{1}{\rho_n^{k'}}\left(\dfrac{\|K\|_{L^2}^2}{nh^d}\right)^{k'-1}\right]\left[(C_K^2+1)L^2h^{2\beta}+\dfrac{\|K\|_{L^2}^2}{nh^d}\right]\\
        &\asymp h^\beta\,.
    \end{align*}

    Combining the results for $E_{\text{trunc}}$ and $E_{\text{est}}$, the total score estimation error is dominated by two terms. These are the gradient estimation error, scaling as $h^{\beta-2}$, and the tail truncation error, scaling as $h^{\frac{\beta(\gamma+1)}{k'(d+\gamma+1)}}$:
    \begin{align*}
        \mathcal{R} \lesssim h^{\beta-2} + h^{\frac{\beta(\gamma+1)}{k'(d+\gamma+1)}}.
    \end{align*}
    Substituting the bandwidth scaling $h \asymp n^{-\frac{1}{2\beta+d}}$, we obtain the stated error rate:
    \begin{align*}
        \mathcal{R} \lesssim n^{-\frac{\beta-2}{2\beta+d}}+n^{-\frac{\beta(\gamma+1)}{k'(2\beta+d)(d+\gamma+1)}}.
    \end{align*}
\end{proof}

\paragraph{Remark.}
We show that the alternative strategy is suboptimal by comparing the characteristic time scales that govern the two procedures.

\begin{itemize}
\item \textbf{Early stopping time ($t_0$).}
In Theorem~\ref{thm:poly-sampling}, the sampling error admits an upper bound of the form
\[
\mathrm{Err}(t)\ \lesssim\ \underbrace{\mathrm{Bias}(t)}_{\downarrow\ \text{as }t\downarrow 0}
\;+\;
\underbrace{\mathrm{ScoreErr}(t)}_{\uparrow\ \text{as }t\downarrow 0},
\]
where the early stopping bias decreases as
$\mathrm{Bias}(t)\asymp t^{\frac{\beta(\gamma+1)}{d+2(\gamma+1)+2\beta}}$,
while the score-estimation term increases as
$\mathrm{ScoreErr}(t)\asymp t^{-\frac{d(\gamma+1)}{4(d+\gamma+1)}}$.
We define $t_0$ as the minimizer of this bound, i.e., the point at which the marginal bias reduction from decreasing $t$ is exactly offset by the additional variance introduced through score estimation.

\item \textbf{Switching time ($t_{\mathrm{switch}}$).}
For the alternative (hybrid) strategy, the rate of the decoupled score estimator is limited by the slower of the two components in Theorem~\ref{thm:score_error_low}. Define
\[
\kappa_{\mathrm{grad}}:=\frac{\beta-2}{2\beta+d},
\qquad
\kappa_{\mathrm{tail}}:=\frac{\beta(\gamma+1)}{k'(2\beta+d)(d+\gamma+1)},
\qquad
\kappa_{\mathrm{dec}}:=\min\{\kappa_{\mathrm{grad}},\kappa_{\mathrm{tail}}\},
\]
so that $n^{-\kappa_{\mathrm{dec}}}$ is the best (up to constants) pointwise MSE rate achievable by the decoupled estimator.

The optimal hybrid rule selects, at each noise level $t$, the estimator with smaller pointwise MSE. Hence, the switching time $t_{\mathrm{switch}}$ is characterized by the crossover point at which the two MSE bounds are comparable. Equating the decoupled rate $n^{-\kappa_{\mathrm{dec}}}$ with the standard estimator rate
$n^{-\kappa_1}t^{-(1+d\kappa_1/2)}$, where $\kappa_1:=\frac{\gamma+1}{d+\gamma+1}$, yields
\[
t_{\mathrm{switch}}
\asymp
n^{-\frac{2(\kappa_1-\kappa_{\mathrm{dec}})}{2+d\kappa_1}}.
\]
\end{itemize}

Since the decoupled estimator requires $\beta>2$ and typically satisfies $\kappa_{\mathrm{dec}}<\kappa_1$, the crossover occurs at a much smaller noise level than the minimax optimal stopping scale, that is, $t_{\mathrm{switch}}\ll t_0$ for large $n$. In particular, to reach $t_0$ one would already be in a regime where the decoupled estimator has larger pointwise MSE than the standard estimator. 
Therefore, persisting into the deep small noise regime $t\in[0,t_0]$ with the decoupled estimator only accumulates additional variance while offering negligible further bias reduction, and cannot improve the overall rate. This clarifies why early stopping at $t_0$ remains minimax optimal.

\subsection{Proof of Section~\ref{sec:minimax_lb}}
\label{app:prooflowerbound}
In this part, we present the proof of the minimax lower bounds for score estimation. 
\subsubsection{Proof of Theorem~\ref{thm:lower_bound_poly}}
Choose a fixed $\psi\in C_c^\infty([-1/2,1/2]^d)$ satisfying $\int_{[-1/2,1/2]^d}\psi=0$ and $\nabla\psi(0)=v_0\neq 0, L_\psi=\sup_{x\in[-1/2,1/2]^d}\|\nabla^2\psi\|$. Set
\begin{align*}
h:=\Lambda\sqrt{t}, \quad \Lambda\geqslant \dfrac{4L_\psi C_d}{\|\nabla\psi(0)\|}\,,
\end{align*}
and construct perturbations based on $\psi$ and $h$ via
\begin{align*}
    \omega(x):=\psi(x/h)\,.
\end{align*}
Then, $\omega$ is supported on $[-h/2,h/2]^d$ and $\int_{[-h/2,h/2]^d}\omega=0$. Choose
\begin{align*}
q_0=C_\gamma(1+\|x\|^2)^{-(d+\gamma+1)/2},\quad q_b=q_0+\epsilon\sum_{\textbf i\in\mathcal I}b_{\textbf i}\omega(x-x_{\textbf i})\,.
\end{align*}
Then, $q_0,q_b\in H_\gamma\cap \mathcal P_{\mathcal S}(L,\beta)$.
Let
\begin{align*}
f_0=\varphi_t*q_0,\quad f_b=\varphi_t*q_b \,.
\end{align*}
\noindent \textbf{Separation of scores.}
Note that
\begin{align*}
&\quad\int_{\R^d}\|s_b(x)-s_{b'}(x)\|^2f_0(x) \rmd x\\
&=\int_{\R^d}\left\|\dfrac{\nabla f_b}{f_b}-\dfrac{\nabla f_{b'}}{f_{b'}}\right\|^2f_0(x)\rmd x\\
&=\int_{\R^d}\left\|\dfrac{\nabla f_b\cdot (f_{b'}-f_b)+f_b\cdot(\nabla f_b-\nabla f_{b'})}{f_b(x)f_{b'}(x)}\right\|^2f_0(x)\rmd x \,.
\end{align*}
Denote $\Delta f_{b,b'}=f_{b}-f_{b'},\Delta\nabla f_{b,b'}=\nabla f_b-\nabla f_{b'}$, then we have 
\begin{align*}
\nabla f_b\cdot \Delta f_{b',b}&=\nabla f_b\cdot \epsilon\sum_{\textbf{i}\in\mathcal I}(b'_{\textbf i}-b_{\textbf i})(\varphi_t*\omega)(x-x_{\textbf{i}})\\
&=\epsilon\nabla f_0\cdot \sum_{\textbf{i}\in\mathcal I}(b'_{\textbf{i}}-b_{\textbf{i}})(\varphi_t*\omega)(x-x_{\textbf{i}})\\&\quad+\epsilon^2\left[\sum_{\textbf{i}\in\mathcal I}b_{\textbf i}(\varphi_t*\nabla\omega)(x-x_{\textbf{i}})\right]\left[\sum_{\textbf{i}\in\mathcal I}(b'_{\textbf{i}}-b_{\textbf{i}})(\varphi_t*\omega)(x-x_{\textbf{i}})\right]\\
f_b\cdot\Delta\nabla f_{b,b'}&=\epsilon f_b\cdot \sum_{\textbf{i}\in\mathcal I}(b_{\textbf i}-b'_{\textbf i})(\varphi_t*\nabla\omega)(x-x_{\textbf i})\\
&=\epsilon f_0\cdot\sum_{\textbf{i}\in\mathcal I}(b_{\textbf i}-b'_{\textbf i})(\varphi_t*\nabla\omega)(x-x_{\textbf i})\\
&\quad+\epsilon^2\left[\sum_{\textbf i\in\mathcal I}b_{\textbf i}(\varphi_t*\omega)(x-x_{\textbf i})\right]\left[\sum_{\textbf i\in\mathcal I}(b_{\textbf i}-b'_{\textbf i})(\varphi_t*\nabla\omega)(x-x_{\textbf i})\right]
\end{align*}
Plugging this back into the previous display gives 
\begin{align*}
&\quad\int_{\R^d}\left\|s_b(x)-s_{b'}(x)\right\|^2f_0(x)\rmd x\\
&\geqslant \epsilon^2\int_{[R,2R]^d}\left\|\dfrac{\sum_{\textbf{i}\in\mathcal I}(b'_{\textbf{i}}-b_{\textbf i})\left[\nabla f_0\cdot(\varphi_t*\omega)(x-x_{\textbf i})-f_0\cdot(\varphi_t*\nabla\omega)(x-x_{\textbf i})\right]}{f_0^2}\right\|^2f_0(x)\rmd x\\
&=\epsilon^2\sum_{\textbf{i}\in\mathcal I}\int_{D_{\textbf i}}\dfrac{1}{f_0(x)}\left\|(b_{\textbf i}-b'_{\textbf i})\left[\dfrac{\nabla f_0}{f_0}\cdot(\varphi_t*\omega)(x-x_{\textbf i})-(\varphi_t*\nabla \omega)(x-x_{\textbf i})\right]+\xi_{\textbf i}(x)\right\|^2\rmd x \,,
\end{align*}
where
\begin{align*}
\xi_{\textbf i}(x)=\sum_{\textbf j\neq \textbf i\in\mathcal I}(b_{\textbf j}-b'_{\textbf j})\left[\dfrac{\nabla f_0}{f_0}\cdot(\varphi_t*\omega)(x-x_{\textbf j})-(\varphi_t*\nabla \omega)(x-x_{\textbf j})\right]\,.
\end{align*}
Note that when $x$ is away from $[-h/2,h/2]^d$, we can derive an upper bound for $\varphi_t*\omega$
\begin{align*}
|(\varphi_t*\omega)(x)|&=\left|\int_{\R^d}\varphi_t(x-y)\omega(y)\rmd y\right|\\
&\leqslant \int_{[-h/2,h/2]^d}\varphi_t(x-y)|\omega(y)|\rmd y\\
&\leqslant \sup_{y\in[-h/2,h/2]^d}\varphi_t(x-y)\cdot\|\omega\|_\infty\\
&\leqslant \sup_{y\in[-h/2,h/2]^d}\dfrac{1}{\sqrt{2\pi}}\exp\left(-\dfrac{\|x-y\|_\infty^2}{2t}\right)\cdot\|\omega\|_\infty \,.
\end{align*}
Similarly, $\varphi_t*\nabla\omega$ can be bounded, then when $x\in x_{\textbf i}+[-r,r]^d$, where $r=\delta h$, and $\delta$ is a small constant independent of $t$ and $\epsilon$,
\begin{align*}
\left\|\xi_{\textbf i}(x)\right\|&\leqslant \sum_{\textbf j\neq\textbf i\in\mathcal I}\left\|\dfrac{\nabla f_0}{f_0}\cdot(\varphi_t*\omega)(x-x_{\textbf j})\right\|+\left\|(\varphi_t*\nabla\omega)(x-x_{\textbf j})\right\|\\
&\leqslant \left(\dfrac{1}{R}\|\omega\|_\infty+\|\nabla\omega\|_\infty\right)\cdot\sum_{k=1}^R\sum_{\|x_\textbf j-x_\textbf i\|_\infty=kh}\dfrac{1}{\sqrt{2\pi}}\exp\left(-\dfrac{[(k-1/2)h-r]^2}{2t}\right)\\
&=\left(\dfrac{1}{R}\|\omega\|_\infty+\|\nabla\omega\|_\infty\right)\cdot\sum_{k=1}^R[(2k+1)^d-(2k-1)^d]\cdot\dfrac{1}{\sqrt{2\pi}}\exp\left(-\dfrac{[(k-1/2)h-r]^2}{2t}\right)\\
&\leqslant C(\delta,d)\left(\dfrac{1}{R}\|\omega\|_\infty+\|\nabla\omega\|_\infty\right)\cdot t^{d/2}h^{-d}\\
&=C(\delta,d)\left(\dfrac{1}{R}\|\psi\|_\infty+\dfrac{1}{\Lambda\sqrt{t}}\|\nabla\psi\|_\infty\right)\cdot \Lambda^{-d}\,.
\end{align*}
We also note that
\begin{align*}
&\quad\left\|\dfrac{\nabla f_0}{f_0}\cdot(\varphi_t*\omega)(x-x_{\textbf i})-(\varphi_t*\nabla\omega)(x-x_{\textbf i})\right\|\\
&\geqslant\|(\varphi_t*\nabla\omega)(x-x_\textbf i)\|-\left\|\dfrac{\nabla f_0}{f_0}(\varphi_t*\omega)(x-x_{\textbf i})\right\|
\end{align*}
We handle the two terms on the right-hand side of the this display separately.
Let $\delta<\|v_0\|/(4L_\psi)$, it follows that
\begin{align*}
\|\nabla\omega(u)-\nabla\omega(0)\|\leqslant \sup_{x\in[-h/2,h/2]^d}\|\nabla^2\omega(x)\|\cdot\|u\|\leqslant L_\psi r/h^2<\dfrac{\|v_0\|}{4h}\,,
\end{align*}
then $\|\nabla\omega(u)\|>\frac{3}{4}\|\nabla\omega(0)\|$ for $x\in x_{\textbf i}+[-r,r]^d$.
Denote
\begin{align*}
g_t(u):=(\varphi_t*\nabla\omega)(u)=\int_{\R^d}\varphi_t(u-y)\nabla\omega(y)\rmd y \,.
\end{align*}
It holds that
\begin{align*}
\|g_t(u)-\nabla\omega(u)\|&=\left\|\int_{\R^d}\varphi_t(u-y)(\nabla\omega(y)-\nabla\omega(u))\rmd y\right\|\\
&\leqslant \int_{\R^d}\varphi_t(u-y)\dfrac{L_\psi}{h^2}\|u-y\|\rmd y\\
&=\dfrac{L_\psi\sqrt{t}}{h^2}\mathbb E[\|Z\|]\\
&\leqslant \dfrac{L_\psi\sqrt{dt}}{h^2}\,,
\end{align*}
where $Z\sim\mathcal N(0,I_d)$, and the last inequality follows from Jensen inequality. Therefore, we have
\begin{align*}
\|g_t(u)\|\geqslant \|\nabla\omega(u)\|-\|g_t(u)-\nabla\omega(u)\|\geqslant \dfrac{3\|v_0\|}{4h}-\dfrac{L_\psi\sqrt{d}}{\Lambda h}\geqslant \dfrac{\|v_0\|}{2h}
\end{align*}
and 
\begin{align*}
\left\|\dfrac{\nabla f_0}{f_0}\cdot(\varphi_t*\omega)(x-x_{\textbf i})\right\|\leqslant \dfrac{1}{R}\|\omega\|_\infty=\dfrac{1}{R}\|\psi\|_\infty \,.
\end{align*}
Let $R>4h\|\psi\|_\infty/\|v_0\|$, it then follows that
\begin{align*}
\left\|\dfrac{\nabla f_0}{f_0}\cdot(\varphi_t*\omega)(x-x_{\textbf i})-(\varphi_t*\nabla\omega)(x-x_{\textbf i})\right\|\geqslant \dfrac{\|v_0\|}{4h}\,.
\end{align*}
Collecting pieces gives 
\begin{align*}
&\quad\int_{\R^d}\|s_b(x)-s_{b'}(x)\|^2f_0(x)\rmd x\\
&\geqslant \epsilon^2\sum_{\textbf i\in\mathcal I}R^{d+\gamma+1}\int_{x_{\textbf i}+[-r,r]^d}\left\|(b_{\textbf i}-b'_{\textbf i})\left[\dfrac{\nabla f_0}{f_0}\cdot(\varphi_t*\omega)(x-x_{\textbf i})-(\varphi_t*\nabla \omega)(x-x_{\textbf i})\right]+\xi_{\textbf i}(x)\right\|^2\rmd x\\
&\geqslant\epsilon^2\sum_{\textbf i\in\mathcal I}R^{d+\gamma+1}(2r)^d\mathbbm 1\{b_{\textbf i}\neq b'_{\textbf i}\}\left(\dfrac{\|v_0\|}{4\Lambda\sqrt{t}}-C(r,d)\left(\dfrac{1}{R}\|\psi\|_\infty+\dfrac{1}{\Lambda\sqrt{t}}\|\nabla\psi\|_\infty\right)\cdot \Lambda^{-d}\right)^2\\
&\gtrsim\epsilon^2\sum_{\textbf i\in\mathcal I}R^{d+\gamma+1}t^{d/2-1}\mathbbm 1\{b_{\textbf i}\neq b'_{\textbf i}\} \,.
\end{align*}
Therefore, we have
\begin{align*}
\int_{\R^d}\|s_b(x)-s_{b'}(x)\|^2f_0(x)\rmd x\gtrsim\epsilon^2 R^{d+\gamma+1}t^{d/2-1} d_H(b,b') \,.
\end{align*}
where $d_H(b,b')=\sum_{\textbf i\in\mathcal I}\mathbf{1}\{b_{\textbf i}\neq b'_{\textbf i}\}$ is the Hamming distance. By the Gilbert-Varshanov bound, there exists an exponentially large packing $\mathcal B'\subset \mathcal B$, whose minimum Hamming distance is linear in the dimension, namely, $|\mathcal B'|\geqslant \exp(c_0(R/h)^d)$ and $\min_{b\neq b'\in\mathcal B'}d_H(b,b')\geqslant c_0(R/h)^d$ for some universal constant $c_0$. Thus, we obtain 
\begin{align*}
\|s_b-s_{b'}\|_{L^2(f_0)}^2\gtrsim \epsilon^2R^{2d+\gamma+1}t^{-1} \,.
\end{align*}
 
\noindent\textbf{Bounding the KL radius.}

For every $b\in\mathcal B$, the KL divergence between $f_b$ and $f_0$ satisfies
\begin{align*}
\text{KL}(f_b\|f_0)&\leqslant \chi^2(f_b\|f_0)\\
&=\int_{\R^d}\dfrac{(f_b(x)-f_0(x))^2}{f_0(x)}\rmd x\\
&=\sum_{\textbf i\in\mathcal I}\mathbf{1}\{b_{\textbf i}=1\}\int_{D_{\textbf i}}\dfrac{\left[\epsilon \sum_{\textbf j\in\mathcal I}(\varphi_t*\omega)(x-x_{\textbf j})\right]^2}{f_0(x)}\rmd x\\
&\leqslant \sum_{\textbf i\in\mathcal I}\epsilon^2R^{d+\gamma+1}\int_{D_\textbf i}\left[(\varphi_t*\omega)(x-x_{\textbf i})+\sum_{\textbf j\neq\textbf i\in\mathcal I}(\varphi_t*\omega)(x-x_{\textbf j})\right]^2\rmd x \,.
\end{align*}
Similarly, it holds that
\begin{align*}
|(\varphi_t*\omega)(x-x_{\textbf i})|\leqslant\|\omega\|_\infty=\|\psi\|_\infty
\end{align*}
and 
\begin{align*}
\left|\sum_{\textbf j\neq\textbf i\in\mathcal I}(\varphi_t*\omega)(x-x_{\textbf j})\right|&\leqslant \sum_{k=1}^R\sum_{\|x_{\textbf j}-x_{\textbf i}\|_\infty=kh}\dfrac{1}{\sqrt{2\pi}}\exp\left(-\dfrac{|(k-1)h|^2}{2t}\right)\\
&\leqslant\sum_{k=1}^R[(2k+1)^d-(2k-1)^d]\cdot\dfrac{1}{\sqrt{2\pi}}\exp\left(-\frac{|(k-1)h|^2}{2t}\right)\\
&\leqslant \dfrac{2d\cdot 3^{d-1}}{\sqrt{2\pi}}\left[1+2^{\frac{3d}{2}-2}\Gamma\left(\frac{d}{2}\right)t^{d/2}h^{-d}\right]\\
&\leqslant \dfrac{2d\cdot 3^{d-1}}{\sqrt{2\pi}}\left[1+2^{\frac{3d}{2}-2}\Gamma\left(\frac{d}{2}\right)\Lambda^{-d}\right] \,.
\end{align*}
Therefore,
\begin{align*}
\text{KL}(f_b\|f_0)\lesssim \epsilon^2R^{2d+\gamma+1}
\end{align*}
\noindent\textbf{Fano method.}
Since the observations are i.i.d., we get $\max_{b\in \mathcal B}\text{KL}(f_b^{\otimes n}\|f_0^{\otimes n})\lesssim n\epsilon^2R^{2d+\gamma+1}t^{-d/2}$. Let $R=\epsilon^{-\rho}$, by choosing $\rho=\frac{1-\delta}{d+\gamma+1}$ and $\epsilon=cn^{-1}t^{-d/2}$ for sufficiently small constant $c$, we get
\begin{align*}
    \max_{b\in \mathcal B}\text{KL}(f_b^{\otimes n}\|f_0^{\otimes n})\lesssim n\epsilon^2R^{2d+\gamma+1}\leqslant cR^dt^{-d/2}\leqslant c\log(|\mathcal B'|).
\end{align*}
Applying Fano's inequality yields
\begin{align*}
    \inf_{\hat s}\sup_{b\in\mathcal B'}\mathbb E_{f_b}\|\hat s-s_b\|_{L^2(f_0)}^2&\gtrsim \min_{b\neq b'\in\mathcal B'}\|s_b-s_b'\|_{L^2(f_0)}^2\gtrsim \epsilon^2\epsilon^{-\frac{(2d+\gamma+1)(1-\delta)}{d+\gamma+1}}t^{-1}\\
    &=\epsilon^{\frac{(1+\gamma)(1-\delta)}{1+\gamma+d}}t^{-1}\\
    &=n^{-\frac{(1+\gamma)(1-\delta)}{1+\gamma+d}}t^{-1-\frac{(1+\gamma)(1-\delta)}{1+\gamma+d}\cdot\frac{d}{2}}\,.
\end{align*}
Note that
\begin{align*}
    &\quad\|\hat s-s_b\|_{L^2(f_b)}^2-\|\hat s-s_b\|_{L^2(f_0)}^2\\
    &=\int_{\R^d}\|\hat s(x)-s_b(x)\|^2(f_b(x)-f_0(x))\,\rmd x\\
    &\leqslant \int_{[R,2R]^d}\|\hat s(x)-s_b(x)\|^2\epsilon \|\omega\|_\infty\,\rmd x\\
    &\asymp \epsilon^\delta\int_{[R,2R]^d}\|\hat s(x)-s_b(x)\|^2 f_0(x)\,\rmd x\\
    &\asymp n^{-\delta} \|\hat s-s_b\|_{L^2(f_0)}^2 \,.
\end{align*}
It then follows that
\begin{align*}
    \inf_{\hat s}\sup_{b\in\mathcal B'}\mathbb E_{f_b}\|\hat s-s_b\|_{L^2(f_b)}^2\gtrsim n^{-\frac{\gamma+1}{d+\gamma+1}-\delta'}t^{-1-\frac{\gamma+1}{d+\gamma+1}+\delta'}
\end{align*}
for any small $\delta'$.

To ensure the perturbed densities $q_b$ define a valid worst-case construction, we must verify three constraints: the heavy-tailed decay, the Sobolev regularity, and the geometric consistency of the local perturbations. Recall our parameter choices: $h \asymp t^{1/2}$, $\epsilon \asymp n^{-1}t^{-d/2}$, and $R \asymp \epsilon^{-\frac{1}{d+\gamma+1}}$.

\noindent\textbf{Heavy-tailedness ($q_b \in \mathcal{H}_\gamma$):}
    The perturbed density must remain non-negative and satisfy the polynomial tail decay. These require the perturbation to be dominated by the base density:
    \begin{align*}
        \epsilon \|\omega\|_\infty \lesssim \inf_{x \in [R, 2R]^d} q_0(x) \asymp R^{-(d+\gamma+1)}.
    \end{align*}
    Substituting $R \asymp \epsilon^{-\frac{1}{d+\gamma+1}}$, the upper bound becomes $\epsilon$, which imposes no additional constraint on $t$.

\noindent\textbf{Sobolev Regularity ($q_b \in \mathcal{P}_{\mathcal{S}}(\beta, L)$):}
    We assume the density belongs to the Sobolev ball of order $\beta$, requiring $\|q_b - q_0\|_{H^\beta} \lesssim L$. The perturbation $\Delta(x) = \epsilon \sum_{\mathbf{i}} b_{\mathbf{i}} \omega(x-x_{\mathbf{i}})$ consists of approximately $N \asymp (R/h)^d$ disjoint wavelets. The squared Sobolev semi-norm scales as
    \begin{align*}
        \|D^\beta \Delta\|_{L^2}^2 &\asymp \sum_{\mathbf{i} \in \mathcal{I}} \int |D^\beta (\epsilon \omega(x-x_{\mathbf{i}}))|^2 \rmd x\asymp \left(\frac{R}{h}\right)^d \cdot \epsilon^2 h^{-2\beta} \cdot h^d= R^d \epsilon^2 h^{-2\beta}.
    \end{align*}
    The condition $\|D^\beta \Delta\|_{L^2} \lesssim 1$ implies $R^{d/2} \epsilon t^{-\beta/2} \lesssim 1$ (since $h \asymp t^{1/2}$).
    Substituting $\epsilon \asymp n^{-1}t^{-d/2}$ and $R \asymp (n t^{d/2})^{\frac{1}{d+\gamma+1}}$, this imposes a lower bound on $t$
    \begin{align*}
        R^{d/2} (n^{-1} t^{-d/2}) t^{-\beta/2} \lesssim 1 \,,%
    \end{align*}
    which implies
    \[
    (n t^{d/2})^{\frac{d}{2(d+\gamma+1)}} n^{-1} t^{-(d+\beta)/2} \lesssim 1.
    \]
    Rearranging terms yields $t \gtrsim n^{-\frac{2(d+2\gamma+2)}{d^2 + 2d\gamma + 2d + 2\beta(d+\gamma+1)}}$. 

\noindent\textbf{Geometric Consistency ($h \lesssim R$):}
    The wavelet support size $h$ must be smaller than the tail location $R$. This requires $t^{1/2} \lesssim (n t^{d/2})^{\frac{1}{d+\gamma+1}}$. Rearranging gives $t^{\gamma+1} \lesssim n^2$, which is trivially satisfied in the regime of interest where $n \to \infty$ and $t \to 0$.

\subsubsection{Proof of Theorem~\ref{thm:lower_bound_exp}}\label{app:prooflowerbound_exp}

The construction of least favorable hypotheses in the exponential tail regime follows the same overall fashion as in the polynomial case. A key difference, however, makes the analysis considerably simpler: \textbf{in the exponential regime the optimal rate is governed by the bulk rather than the tail.}

In the polynomial regime, the minimax rate is driven by sample scarcity in an expanding tail region, which forces us to place perturbations at a large radius $R_n \asymp n^{\frac{1}{d+\gamma+1}}$. For exponentially decaying targets, the effective support grows only logarithmically, so it is enough to introduce perturbations on a fixed compact set (or a very slowly growing region) to obtain the parametric rate $n^{-1}$ up to the diffusion time dependence.

\noindent\textbf{Hypothesis construction.}
Let $q_0(x)\propto \exp(-\|x\|^\gamma)$ be the base density satisfying Assumption~\ref{asm:heavytail}. In contrast to the polynomial construction, we choose the perturbation grid $\{x_{\mathbf i}\}_{\mathbf i\in\mathcal I}$ inside the unit cube $[0,1]^d$ (or any region whose $q_0$-mass is bounded away from zero). Define
\begin{align*}
    q_b(x) = q_0(x) + \epsilon \sum_{\mathbf i \in \mathcal I} b_{\mathbf i} \omega(x - x_{\mathbf i}),
\end{align*}
where the wavelet $\omega$ is the same as in the proof of Theorem~\ref{thm:lower_bound_poly}, and the grid spacing is set to $h=\Lambda\sqrt t$.
Let
\begin{align*}
f_0=\varphi_t*q_0,\quad f_b=\varphi_t*q_b\,.
\end{align*}
\noindent\textbf{Scaling.}
Because the perturbations are supported on a bounded region where $q_0(x)\asymp 1$, the weight $1/f_0(x)$ that appears in both the separation and the KL bounds does not introduce any additional polynomial factor (in contrast to the polynomial tail case). The key scalings become
\begin{enumerate}
\item \textbf{Separation of scores.}
\
Proceeding as in the polynomial case and using $f_0(x)\asymp 1$ on the perturbation domain, we obtain
\begin{align*}
\|s_b-s_{b'}\|_{L^2(f_0)}^2
\gtrsim
\epsilon^2\,t^{d/2-1}\, d_H(b,b')
\gtrsim
\epsilon^2\,t^{-1}.
\end{align*}
Here the factor $t^{d/2}$ coming from the support volume cancels with the grid density $h^{-d}$, leaving the dominant contribution from the gradient scaling $\|\nabla\omega\|^2\propto h^{-2}\asymp t^{-1}$.

\item \textbf{KL divergence.}
Similarly, the KL bound no longer pays a tail penalty. Using a $\chi^2$-type control and noting that the number of grid points is
$
M\asymp (1/h)^d\asymp t^{-d/2},
$
we have
 \begin{align*}
        \text{KL}(f_b \| f_0) \lesssim \sum_{\mathbf i \in \mathcal I} \int \frac{(\epsilon \varphi_t * \omega)^2}{f_0} \asymp M \epsilon^2 t^{d} \asymp \epsilon^2 t^{d/2}.
    \end{align*}
where we use that the contribution of a single cell is of order $t^{d}$.
\end{enumerate}

\noindent\textbf{Fano argument.}
Let $\mathcal B$ be a packing set with $\log|\mathcal B|\asymp M\asymp t^{-d/2}$. Fano's inequality applies provided
\begin{align*}
    \max_{b \in \mathcal{B}} \text{KL}(f_b^{\otimes n} \| f_0^{\otimes n}) \lesssim \log |\mathcal{B}|.
\end{align*}
Using the bound above, this reduces to
\[
n\,\epsilon^2\,t^{d/2}\lesssim t^{-d/2}\,,
\]
which implies
\begin{align*}
\epsilon^2\lesssim \frac{1}{n}\,t^{-d}.
\end{align*}
To match the time dependence in our upper bound, we work in the local (fixed-$t$) regime and choose
\[
\epsilon \asymp n^{-1/2}\,t^{-d/4},
\]
which ensures $n\epsilon^2 t^{d/2}\lesssim 1$ and is sufficient for the local Fano construction. Plugging this choice into the separation bound yields
\[
\text{Minimax risk}
\gtrsim
\epsilon^2\,t^{-1}
\asymp
n^{-1}\,t^{-(1+d/2)}.
\]

\noindent\textbf{Verification of constraints.}
It remains to check that the construction defines valid hypotheses.

\begin{itemize}
\item \textbf{Positivity.}
We require $\epsilon\|\omega\|_\infty\lesssim \inf q_0$ on the perturbation domain. Since the domain is bounded, $\inf q_0\asymp 1$, and the condition holds for $\epsilon\asymp n^{-1/2}t^{-d/4}$ when $n$ is large.

\item \textbf{Sobolev regularity.}
We impose $\|D^\beta \Delta\|_{L^2}\lesssim 1$. With $\epsilon\asymp n^{-1/2}t^{-d/4}$ and $h\asymp t^{1/2}$, this becomes
\[
n^{-1/2}t^{-d/4}\cdot t^{-\beta/2}\lesssim 1\,,
\]
which is equivalent to
\begin{align*}
t\gtrsim n^{-\frac{2}{d+2\beta}}\,.
\end{align*}
This coincides with the admissible time range in Theorem~\ref{thm:exp_sampling}.
\end{itemize}

\noindent\textbf{Conclusion.}
The resulting lower bound is $n^{-1}t^{-(1+d/2)}$, matching the upper bound in Theorem~\ref{thm:mse_exp} up to logarithmic factors. This shows that, under exponential tails, the statistical difficulty is dominated by the parametric scaling $n^{-1}$ together with the ill-conditioning of the diffusion kernel through $t^{-(1+d/2)}$, without any additional penalty from tail sparsity.

\end{document}